\documentclass[12pt]{amsart}
\usepackage{latexsym, amsmath, amssymb,amsthm,amsopn,amsfonts}
\usepackage{version}
\usepackage{epsfig,graphics,color,graphicx,graphpap}
\usepackage{amssymb}
\usepackage{todonotes}

 \definecolor{skyblue}{rgb}{0.85,0.85,1}

\ifx\pdfoutput\undefined
  \DeclareGraphicsExtensions{.eps}
\else
  \ifx\pdfoutput\relax
    \DeclareGraphicsExtensions{.eps}
  \else
    \ifnum\pdfoutput>0
      \DeclareGraphicsExtensions{.pdf}
    \else
      \DeclareGraphicsExtensions{.eps}
    \fi
  \fi
\fi

\newcommand{\RR}{{\mathbb R}}

\newcommand{\LCB}{\left\{} 
\newcommand{\RCB}{\right\}}

\newcommand{\LC}{\left(}
\newcommand{\LV}{\left|} 
\newcommand{\RC}{\right)}
\newcommand{\RV}{\right|}
\newcommand{\LB}{\left[}
\newcommand{\RB}{\right]}
\newcommand{\LN}{\left\|}
\newcommand{\RN}{\right\|}

\newcommand{\curl}{\operatorname{curl}}
\newcommand{\dv}{\operatorname{div}}
\newcommand{\dist}{\operatorname{dist}}

\newcommand{\J}{\mathbb{J}}

\newcommand{\etae}{\eta_{\varepsilon}}

\newcommand{\logep}{{\LV \log \e \RV}}

\newcommand{\rt}{\tilde{\rho}}

\newcommand{\tQ}{\widetilde{Q}}

\newcommand{\R}{{\bf R}}
\newcommand{\C}{{\mathbb C}}

\newcommand{\Z}{{\mathbb Z}}

\newcommand{\Ee}{E_\e^{\eta_\e}}

\newcommand{\e}{\varepsilon}

\newcommand{\Wdot}{\dot{W}^{-1,1}}

\theoremstyle{plain}

\newtheorem{thm}{Theorem}
\newtheorem{prop}{Proposition}[section]

\newtheorem{lem}[prop]{Lemma}

\theoremstyle{definition}

\newtheorem{rem}{Remark}[section]

\numberwithin{equation}{section}
\newcommand{\thmref}[1]{Theorem~\ref{#1}}

\newcommand{\propref}[1]{Proposition~\ref{#1}}

\def\squarebox#1{\hbox to #1{\hfill\vbox to #1{\vfill}}}

\newcommand{\p}{\partial}
\usepackage{amsxtra}

\title[IGP vortex motion]
{Gross-Pitaevskii vortex motion \\ with critically-scaled inhomogeneities}

\author[M. Kurzke]
{Matthias Kurzke}
\email{matthias.kurzke@nottingham.ac.uk}

\author[J.L. Marzuola]
{Jeremy L. Marzuola}
\email{marzuola@math.unc.edu}

\author[D. Spirn]
{Daniel Spirn}
\email{spirn@math.umn.edu}

\address{Mathematics Department, UNC \\
 Chapel Hill, NC, USA}

\address{Mathematics Department, University of Minnesota \\
 Minneapolis, MN, USA}

\address{School of Mathematical Sciences, University of Nottingham\\Nottingham NG7 2RD, United Kingdom}

\thanks{
 J.L.M. was support in part by NSF Applied Math Grant DMS-1312874 and NSF CAREER Grant DMS-1352353.
 D.S. was supported in part by NSF CAREER grant DMS-0955687 and NSF grant DMS-1516565.  The authors wish to thank Sylvia Serfaty and Rafe Mazzeo for helpful conversations throughout the development of this work and wish
 to thank the  anonymous referees for their comments that helped to improve the paper.}

\begin{document}

\begin{abstract}
We study the dynamics of vortices in an inhomogeneous Gross-Pitaevskii equation $i \p_t u = \Delta u + {1\over \e^2} (p_\e^2(x) - |u|^2)$.  For a unique scaling regime $|p_\e(x) - 1 | = O(\logep^{-1})$, it is shown that vortices can interact  both with the background perturbation and with each other.  Results for associated parabolic and elliptic problems are discussed.  

\end{abstract}
   
\maketitle 

\section{Introduction}

We consider the behavior of vortices in an inhomogeneous Gross-Pitaevskii (IGP) equation  
\begin{equation} \label{eq:igl}
\left\{
\begin{array} {rl}
i \p_t u  = \Delta u + {1\over\e^2} (p^2(x) - |u|^2) u  &  \hbox{ in } \Omega \\
\nu \cdot \nabla u  = 0 & \hbox{ on } \p \Omega
\end{array} \right.  ,
\end{equation}
where $\Omega$ is a bounded, simply connected domain in $\R^2$.  

An important feature of solutions of \eqref{eq:igl} is the potential appearance of vortices, points where $u(x,t):\mathbb{R}^2 \to \mathbb{C}$ has a zero with nontrivial topological degree. We concentrate on the case where there are $n$ vortices of degree $d_j \in \{-1,1\}$ at positions $a_j (t) \in \mathbb{R}^2$ for $j = 1, \dots, n$. We will show below that solutions take the form,
%
\begin{equation*}
u  (x,t) \approx \prod_{j=1}^n \rho_\e(x) \LC { x - a_j(t) \over | x - a_j(t) |} \RC^{d_j}.
\end{equation*}
Here ${ x - a_j(t) \over | x - a_j(t) |}$
is interpreted in the complex sense as 
$ { z - a_j(t) \over | z - a_j(t) |}$ for $z-a_j = x_1-a_{j,1} + i (x_2 - a_{j,2})$
with  $\rho_\e(x) \to 0$ as $|x-a_j| \to 0$ and $|u (x,t)|^2 \approx p^2(x)$ away from the $a_j(t)$'s.   In a  sense identifying the location and degrees of vortices is sufficient to describe the solution $u (x,t)$.   

The inhomogenous Gross-Pitaevskii equation \eqref{eq:igl} arises in several contexts, including the behavior of  superfluid ${}^4$He near a boundary, Bose-Einstein Condensation (BEC) under a very weak confinement potential, and nonlinear optics.   
In the first case it has been observed experimentally that anomalous behavior of   superfluid  ${}^4$He vortex filament motion indicates that  vortex tubes can become  pinned to sites on the material boundary, see \cite{photograph, schwarz1985three, tsubota1993pinning}.  Such pinning can be incorporated heuristically by the introduction 
of a potential $p(x)$ in the Gross-Pitaevskii equation, see \cite{PR} or complex Ginzburg-Landau equations, see \cite{gil1992dynamics}.    In the second case, IGP is a useful model  in a BEC in which the vortices 
interact both with each other and with the trap potential. 
In the third case, IGP is useful in modeling optical vortices.  See for instance \cite{arecchi1991vortices,Arecchi, PR}, which are related to motion of vortices with non-vanishing total charge in inhomogeneous potentials and confinement.

The dynamics of vortices in Gross-Pitaevksii equations have been the subject of much study in past few decades.  The connection between
the vortex motion in superfluids, where $p(x) \equiv 1$, and simplified ODE's was made by Fetter  in \cite{fetter}, and it was shown that  vortices  interact with each through a Coulomb potential.  
On the other hand in Bose-Einstein condensates with nontrivial trapping potentials, $p(x) = O(1)$, it was shown by Fetter-Svidzinsky in \cite{fetter2} that vortices interact solely with the background potential and are carried along level sets of the Thomas-Fermi profile.   

On the other hand it has been observed, both experimentally \cite{freilich2010real, neely2010observation} and numerically \cite{MKFCS}, that  vortices do interact with  other vortices and with the background potential.  In one particularly interesting direction, experimentalists have generated dipoles by dragging non radially symmetric BEC's through  laser obstacles in \cite{neely2010observation} and by the Kibble-Zurek mechanism in \cite{freilich2010real}.  These bound dipoles interact in nontrivial ways with each other and the background potentials.  Reduced ODE models for the dynamics of these vortex configurations can be found in \cite{MTKFCSFH, SKS, SMKS, torres2011dynamics} and are in good agreement with numerical simulations of the full equation \eqref{eq:igl} and with physical experiments, \cite{MTKFCSFH}.    The objective of our work is to place these widely-studied models on a rigorous footing.

We are  interested in the critical asymptotic regime where 
vortices  interact with both the background potential $p(x)$ and each other.  
 In the following, we will take $p (x) = p_\epsilon (x)$ with
\begin{equation} \label{e:pubd}
|1 - p_\e^2 (x) | \lesssim {1\over \logep} 
\end{equation}
and show that asymptotic scalings of $p_\e(x)$ outside of the range \eqref{e:pubd} induce  dynamics that are  dominated  by  only the background potential or by only vortex-vortex interactions.

\subsection{Background potential }

Following Lassoued-Mironescu \cite{LM} we write 
$$u(x,t) = \eta_\e(x) w(x,t),$$ 
where $\eta_\e$ is a minimizer of an energy \eqref{e:etaenergy} and consequently a nontrivial solution to the following elliptic PDE,
\begin{equation}  \label{eq:TFeqn}
\left\{ \begin{array}{rl}
0 = \Delta \eta_\e + {1\over \e^2} \eta_\e ( p_\e^2 - \eta_\e^2 ) & \hbox{ in } \Omega \\
\nu\cdot\nabla \eta_\e = 0 & \hbox{ on } \p \Omega \end{array} \right.,
\end{equation}
then \eqref{eq:igl} is equivalent to
 the Schr\"odinger equation
\begin{equation}  \label{e:weightedGP}
\left\{ \begin{array}{rl}
i \eta^2_\e \p_t w = \dv ( \eta^2_\e \nabla w) + { \eta^4_\e \over \e^2} ( 1 - |w|^2 ) w  & \hbox{ in } \Omega \\
\nu\cdot\nabla w = 0 & \hbox{ on } \p \Omega \end{array} \right. .
\end{equation}
The local and global well-posedness of \eqref{e:weightedGP} for fixed $\e$ can be established using a variation of the argument in Brezis-Gallouet \cite{BG}.

Associated to \eqref{e:weightedGP} is a weighted Hamiltonian,
\[
e_\e^{\eta_\e} (w) := {1\over 2} \eta_\e^2 \LV \nabla w \RV^2 + {\eta_\e^4 \over 4 \e^2} \LC 1 - |w|^2 \RC^2 ;
\]
consequently, $e_\e^1(w)$ is the standard gauge-less Ginzburg-Landau energy density as in \cite{BBH}.  
We set 
\[
E_\e^{\eta_\e} (w):= \int_{{ \Omega}} e_\e^{\eta_\e}(w) dx
\]
to be the total energy of the mapping $w$ associated to background potential $\eta_\e$.  We will be interested in an asymptotic range of $p_\e(x)$'s such that
the resulting energy feels both the background potential and the other vortices at the same order.    As will be elicited below, our regime corresponds to having 
the background $\eta_\e$ expanded 
in the following way: $\eta^2_\e (x) = 1 + {Q_\e(x) \over \logep}$ with $Q_\e(x) \to Q_0(x)$ as $\e \to 0$.



\subsection{Conservation Laws}

There are a set of  conservation laws for \eqref{e:weightedGP}, and to write them down efficiently, we introduce 
some notation.
For $b, c \in \mathbb{C}$ set $(\cdot, \cdot) = {1\over 2} \LC b \overline c + \overline{b} c \RC$ and let us define 
\begin{align*}
j(u) := ( i u , \nabla u) & \equiv \hbox{ supercurrent }  , \\
J(u)  := \det \nabla u = {1\over 2} \operatorname{curl} j(u)& \equiv \hbox{ Jacobian },
\end{align*}
then solutions satisfy the following differential identities:
\begin{align} 
 \label{eq:consmass}
{\eta_\e^2 \over 2} \p_t \LC |w|^2  - 1 \RC & 
= \dv \LC \eta^2_\e j(w) \RC,  \\  
\label{eq:consenergy}
\p_t e_\e^{\eta_\e}(w)  & = \dv ( \eta^2_\e ( \nabla w, \p_t w ) ) .
\end{align}
A third conservation law holds for the Jacobian, we which we now define, tensorially. If the matrix $\J$ is given by $\J_{ij}$ with $\J_{2 1} = - \J_{1 2} = 1$ and $\J_{11}= \J_{22} = 0$ then $\curl {F} = \J_{k \ell} \p_\ell F_k$
and if $J(w)  = {1\over2} \curl j(w)$, then following Jerrard-Smets \cite{JSm} we have after a lengthy calculation
\begin{equation}\label{modJacev}
  \p_t J (w) = \J_{\ell j} \p_j \LB {1\over \eta^2_\e} \p_k \LC \eta^2_\e \p_\ell w, \p_k w \RC \RB
 - \J_{\ell j} \p_j \LB {\eta^2_\e } \p_\ell {\LC 1 - |w|^2 \RC^2 \over 4 \e^2}  \RB .
\end{equation}
Multiplying by $\phi \in C_0^2(\Omega)$ yields a conservation law observed in \cite{JSm}
\begin{equation}\label{e:modJacev2}
\begin{split}
 \int_{\Omega} \phi \p_t J (w) dx & = \int_{\R^2} \J_{\ell j} \p_{j} \p_{k } \phi     \LC \p_k w, \p_\ell w \RC dx \\
 & \quad - \int_{\Omega} \J_{\ell j} \p_{j} \phi \LB  { \p_k \eta^2_\e \over \eta^2_\e} \LC \p_\ell w, \p_k w \RC  +  {\p_\ell \eta^2_\e } {\LC 1 - |w|^2 \RC^2 \over 4 \e^2}  \RB dx .
\end{split}
\end{equation}

The dynamics of a vortex can be inferred by choosing suitable test functions $\phi$ in \eqref{e:modJacev2} which elicit the vortex positions.  Since we expect
the Jacobian $J(w)$ to be roughly equal to a delta function, localized at the site of each vortex $\alpha$, we can choose $\phi = x_m \chi(x - \alpha)$, $m\in \{1,2\}$ for smooth cutoff function $\chi(x) = 1$ in $B_r$ and compactly supported in $B_{2r}$.  Integrating the left hand side over the support yields, formally,
$\pi \dot{\alpha}$.  On the right hand side, the first term has support in an annulus about each vortex, due to the structure of the test function, and since the  energy concentrates
outside of the support of the annulus,  one sees that it generates  
 an $O(1)$ vortex-vortex interaction term.    The second term on the right hand side has support inside of a ball about each vortex, and since the energy tensor is large in this region \cite{KSEquipartition}, it generates  an $O( \logep | \nabla \log \eta_\e|)$  vortex-potential interaction term.  Therefore, for a vortex to interact with both the background potential and  other vortices,
 it requires that $|\nabla \log \eta_\e | = O( \logep^{-1})$ for both terms to be of the same order.\footnote{The third term on the right hand side of \eqref{e:modJacev2}} is negligible. 
 
 In fact if $|\nabla \log \eta_\e | \ll \logep^{-1}$ then the vortex dynamics will be dominated by the interaction with the background potential, as in \cite{JSm}, whereas if $| \nabla \log \eta_\e | \gg \logep^{-1}$ then vortex dynamics will be driven by its interaction with the current generated by fellow vortices, see \cite{CJ}.  
 
 \subsection{Weak Topology}

In order to measure the distance of the vortices to the site of the expected location of the vortices given by the ODE \eqref{e:ODE}, we use 
the flat norm $\Wdot(\Omega)$.  We denote the norm $\Wdot$ unless  this topology is used on a subdomain, such as 
$$\| \mu \|_{\Wdot(B_r(\alpha))} = \sup_{\| \nabla \phi \|_{L^\infty(B_r(\alpha))} \leq 1}  \LCB \int_{B_r(\alpha)} \phi d \mu  \hbox{ such that } \phi \in W^{1,\infty}_0(B_r(\alpha)) \RCB. $$

We use a helpful estimate for concentrations in the flat norm, see Brezis-Coron-Lieb \cite{BCL}.  Define 
\begin{equation} \label{eq:rho}
r_\alpha :=  {1\over 8} \min \left\{  \min_{i \neq j} |\alpha_i - \alpha_j|, \min_{j} \operatorname{dist}(\alpha_j, \p \Omega) \right\} ,
\end{equation}
then the following holds:
\begin{lem}
Suppose $\alpha_k, \xi_k$ are points in $\Omega$.  If 
$$\LN \sum d_k \delta_{\alpha_k} - d_k \delta_{\xi_k} \RN_{\Wdot} \leq {1\over 4} r_\alpha,$$ 
then
\[
\LN \sum d_k \delta_{\alpha_k} - d_k \delta_{\xi_k} \RN_{\Wdot} = \sum_{k} | \alpha_k - \xi_k |.  
\]
\end{lem}
\begin{proof}
For a proof see \cite{JSm} for example.
\end{proof}
Consequently, the flat norm provides a good way to measure the location of singularities.  Since the 
Jacobian, $J(w)$, converges to a sum of weighted delta functions, we will use $\Wdot$ topology as a convenient metric to track vortex trajectories.

Associated to a set of locations $\alpha_j \in \Omega$ and degrees $d_j \in \{\pm 1\}$ we can define a \emph{canonical harmonic map, $w_*$ that describes the limiting harmonic map with prescribed 
singularities, 
\[
w_*(\alpha, d) \equiv e^{i\psi_*} \prod_{j=1}^n \LC { x - \alpha_j \over |x - \alpha_j|} \RC^{d_j}
\]
where harmonic $\psi_*$ is chosen so that $\p_\nu w_*(\alpha,d) = 0$ on $\p \Omega$.  We will use $w_*$ when there is no ambiguity.
}

\subsection{Energy Expansion}
\label{intro_ee}

Given a collection of $n$ vortices with centers $\alpha = \{\alpha_1, \dots, \alpha_n \}$, the Hamiltonian $E_\e^{\eta_\e}(w)$ can be expanded out to second order with $E_\e^{\eta_\e}(w)= H_\e(\alpha) + o_\e(1)$ where 
\[
 H_\e(\alpha) := \sum_{j=1}^n\LC \pi \logep + \pi Q_0(\alpha_j) + \gamma_0 \RC + W(\alpha,d) ,
\]
which will be established in the sense of $\Gamma$-convergence in Proposition~\ref{prop:Gammac2}.
Here $n$ will be used throughout to represent the number of vortices, $Q_0(x)$ is the limiting rescaled background perturbation, and $W(\alpha, d)$ is the renormalized energy that arises in the work of Bethuel-Brezis-H\'elein \cite{BBH} \footnote{The authors of \cite{BBH} study the case of  Dirichlet boundary conditions.  See 
the discussion in the appendix of \cite{SpirnGP} for the derivation of $W(\alpha,d)$ with Neumann boundary conditions.}, given by
\begin{align*}
W(\alpha, d) & = \lim_{r \to 0} \LB \min_{\substack{v \in H^1 (\Omega_r(\alpha);S^1) \\  \deg(v;\p B_r(\alpha_j)) =d_j} }
\int_{\Omega_r(\alpha)} {1\over2 }  \LV \nabla v\RV^2 dx  - \sum_{j=1}^n \pi  \log {1\over r} \RB   \\
& = - \pi \sum_{j\neq k} d_j d_k \log |\alpha_j - \alpha_k| + \hbox{ boundary terms},
\end{align*}
 where $\Omega_r (\alpha) \equiv \Omega \backslash \cup_{j} B_r(\alpha_j)$.
In order to express the boundary terms, we define $G(\cdot,\alpha,d) = \sum_{j=1}^n G(\cdot, \alpha_j, d_j)$ where
\[
\Delta G(\cdot, \alpha_j , d_j) = 2 \pi d_j \delta_{\alpha_j} \hbox{ in } \Omega \hbox{ with } G = 0 \hbox{ on } \p \Omega.
\]
If we set $F(x, \alpha_j) = G(x,\alpha_j, d_j) -  \log |x - \alpha_j|$, then we can fully express $W(\alpha, d)$ as 
\begin{equation}\label{eq:NeumannW}
W(\alpha, d) = -  \pi \sum_{j \neq k} d_j d_k \log| \alpha_j - \alpha_k| +  \pi \sum_{j, k} d_j F(\alpha_j, \alpha_k),
\end{equation}
see \cite{BBH} and the appendix of \cite{SpirnGP}.


\subsection{Discussion}

Rigorous results on vortex dynamics for Gross-Pitaevskii were established when $p(x) \equiv 1$  by Colliander-Jerrard \cite{CJ} and also Lin-Xin \cite{LX} which showed that vortices satisfy the Kirchoff-Onsager ODE for Euler point vortices.   When $p(x) = O(1)$ there is a recent result of Jerrard-Smets \cite{JSm} that showed that vortices travel along level sets of the Thomas-Fermi profile.  
In the parabolic setting there were rigorous results by Lin  \cite{LinParabolic} and Jerrard-Soner \cite{JSDynamics} when $p(x) \equiv 1$ and by Jian-Song \cite{jiansong} when $p = O(1)$. Mixed dynamics with $p=O(1)$ and further forcing terms were discussed by Serfaty-Tice \cite{ST}.
 A variational proof of the parabolic dynamics for $p \equiv 1$ was given by Sandier-Serfaty \cite{SSGamma}.



\subsection{Results}

We wish to describe the dynamics of vortices in \eqref{e:weightedGP}.  If the vortices are located at positions $\alpha_j(t)$ with degree $d_j \in \{-1,1\}$, then the vortices move via the ODE,
 \begin{equation} \label{e:ODE}
\pi d_k \dot{\alpha}_{k} (t) = \nabla^\perp_{\alpha_{k}} H_0 (\alpha, Q_0),
\end{equation}
where $\nabla^\perp = (\p_2, -\p_1)^T$ and
\begin{equation} \label{e:Hdef}
 H_0 (\alpha, Q_0) = W(\alpha, d) + \pi \sum_{j=1}^n Q_0(\alpha_j) .
\end{equation}

Given the ODE generated by \eqref{e:ODE}, we can define the collision time, 
\[
T_{col} := \sup_{t \geq 0} \{ r_{\alpha (s)} > 0 \hbox{ for all } 0 \leq s \leq t \} .
\]
Then for any $0 \leq T < T_{col}$ we can define 
\begin{equation} \label{eq:rho.min}
r_{min} := \min_{t \in [0,T] } r_{ \alpha  (t)}
\end{equation}
which is the minimum vortex-vortex or vortex-boundary distance until time $T$.  Clearly $r_{min}>0$.

Our main result is the following that captures both vortex-vortex and vortex-potential interactions.   An important hypothesis will be 
that the initial data has control of the excess energy 
\begin{equation} \label{e:Ddef}
D_\e(w,\alpha) \equiv E_\e^{\etae}(w) - H_\e(\alpha, Q_0).
\end{equation}
We will shorten the notation to $D_\e$ when the $w$ and $\alpha$ are readily apparent.

\medskip

\begin{thm} \label{t:vortexdynamics}
Let $p_\e^2 = 1 +{\rho_\e(x) \over \logep}$ where the $\rho_\e$ satisfies hypotheses of Proposition~\ref{p:ellconv} with $k \geq 4$.
Let $\{\alpha_{j}^0, d_j\}$ is a configuration of vortices such that $ r_{\alpha_0} > 0 $, $d_j \in \{-1, 1\}$ and suppose $w_\e^0, 0 < \e <1$ satisfies the well-preparedness hypotheses 
\begin{equation} \label{e:wellpreparedness}
\| J(w_\e^0) - \pi \sum d_j \delta_{\alpha_j^0} \|_{\dot{W}^{-1,1}}= o_\e(1) \hbox{ and } D_\e(w_\e^0, \alpha^0) = o_\e(1) .
\end{equation}
If $w_\e(t)$ is a solution to \eqref{eq:igl} with initial data $w_\e^0$ 
then there exists a time $T > 0$ 
such that for all $t \in [0,T]$, 
\[
\LN J(w_{\e}(\cdot , t ) ) - \pi \sum d_j \delta_{\alpha_j(t)} \RN_{\dot{W}^{-1,1}(\Omega)} \to 0 
\]
 as $\e\to 0$, where the $\alpha_j(t)$ are defined by
\eqref{e:ODE}, and $T$ is independent of $\e$.

\end{thm}

\begin{rem}
An important function satisfying the conditions of Proposition \ref{p:ellconv} below is simply $\rho_\e = \rho_0 \in H^{k+1} (\Omega)$ is any fixed function such that $\nabla \rho_0 $ has compact support inside $\Omega$.
\end{rem}



One important example that arises from the Theorem~\ref{t:vortexdynamics} are dipoles that are scattered by inhomogeneities, see Remark~\ref{r:extensions} below and Section~\ref{s:numerics} for some illustrative numerical simulations.

Under slightly less regularity assumptions, we have a similar result for the gradient flow
\begin{equation}\label{eq:introGF}
\frac1\logep \p_t u_\e = \Delta u_\e + \frac1{\e^2} (p_\e^2-|u_\e|^2) u_\e
\end{equation}
with the limit equation
 \begin{equation} \label{e:introGFODE}
\pi \dot{\alpha}_{k} (t) = -\nabla_{\alpha_{k}} H_0 (\alpha, Q_0),
\end{equation}
\begin{thm}\label{t:introGF}
Let $p_\e^2 = 1 +{\rho_\e(x) \over \logep}$ where the $\rho_\e$ satisfies hypotheses of Proposition~\ref{p:ellconv} with $k \geq 3$.
If $\{\alpha_{j}^0, d_j\}$ be a configuration of vortices such that $d_j \in \{-1, 1\}$ and if $u_\e^0, 0 < \e <1$ satisfies \eqref{eq:introGF} 
and $w_\e=\frac{u_\e}{\eta_\e}$ satisfies the well-preparedness hypotheses 
\begin{equation} \label{e:wellpreparedness_gradflow}
\| J(w_\e^0) - \pi \sum d_j \delta_{\alpha_j^0} \|_{\dot{W}^{-1,1}}= o_\e(1) \hbox{ and } D_\e(w_\e^0, \alpha^0) = o_\e(1) \hbox{ and } r_{\alpha_0} > 0,
\end{equation}
then there exists a time $T > 0$ 
such that for all $t \in [0,T]$, 
\[
\LN J(w_{\e}(\cdot , t ) ) - \pi \sum d_j \delta_{\alpha_j(t)} \RN_{\dot{W}^{-1,1}(\Omega)} \to 0 
\]
 as $\e\to 0$, where the $\alpha_j(t)$ are defined by
\eqref{e:introGFODE}.
\end{thm}


\begin{rem}
\label{r:extensions}
There are straightforward adaptations of Theorem~\ref{t:vortexdynamics} to other contexts, including:
\begin{itemize}
\item The Gross-Pitaevskii equation \eqref{eq:igl} on  $\Omega \equiv \R^2$.  This can be achieved by combining the 
methods here with the arguments in \cite{BJS, JSp3}.   In this case vortices move according to the same ODE \eqref{e:ODE}; however, 
the renormalized energy $W(a,d)$ is  the classical Coulomb potential,  
\[
W(\alpha, d) = -  \pi \sum_{j \neq k} d_j d_k \log|\alpha_j - \alpha_k|.
\]

\item Mixed Ginzburg-Landau equations of the form 
\[
\LC {1\over \logep} \p_t + i \p_t  \RC u_\e = \Delta u_\e + {1\over \e^2} \LC p^2_\e(x) - |u_\e|^2 \RC u_\e
\]
as studied in \cite{KMMS,Miot} which results in a modified ODE 
\[
 \pi \dot{\alpha}_k (t) - \pi d_k (e_3 \times \dot{\alpha}_k (t) ) = - \nabla_{\alpha_k} H_0(\alpha,Q_0). 
\]
\end{itemize}
\end{rem}

In Section~\ref{sec:gamma} we present  first and second order Gamma-convergence results for the energy $E_\e^{\eta_\e}(w_\e)$, see Proposition~\ref{prop:Gammac1} and Proposition~\ref{prop:Gammac2}, respectively.  These  results are similar to those found in \cite{AP}, albeit ported to the case with critical inhomogeneities.

\subsection{Properties of $\eta_\e = \eta_\e(p_\e)$}

We collect here some information required on the behavior of $\eta_\e$, as related to the background fluctuations $p_\e$.  The discussion of the proof will be given in Appendix \ref{s:etacontrol} and is based off of standard elliptic theory estimates, so we simply state the results that we require for the vortex dynamics here.  

We write
\begin{equation}
\label{Qe-def}
p_\e^2 = 1+\frac{\rho_\e}{|\log\e|}, \ \ \eta_\e^2 =1+\frac{Q_\e}{|\log\e|}.
\end{equation}
We will assume that $\rho_\e  \to \rho_0(x)$ in $H^k (\Omega)$ with $k$  made precise below.  Using results such as \cite{dS2} for the Dirichlet problem, we expect that $p_\e$ and $\eta_\e$ should be quite close.  We make precise the nature in which that is true in the following Proposition, which will be proved in the Appendix.

We remark here that the elliptic analysis required for the convergence that will implement in the Appendix is somewhat non-trivial, as the overall ellipticity of the underlying nonlinear elliptic problem is going to $0$ in the limit as $\varepsilon \to 0$.  Hence, estimates must be done with some care, especially to understand regularity up to the boundary of our domain.  

\begin{prop}
\label{p:ellconv}
Let $\rho_\e$ be such that $\rho_\e \to \rho_0 (x)$ in $H^k (\Omega) $, $\nabla \rho_\e$ is compactly supported strictly on the interior of $\Omega$ for all $\e$ and  $\e^2 \rho_\e \to 0$ in $H^{k+1}$.  Then, for $Q_\e$ as defined in \eqref{eq:TFeqn}, \eqref{Qe-def} satisfying Neumann boundary conditions on $\p \Omega$, we have $Q_\e \to \rho_0$  in $H^k (\Omega)$.  In particular, we have $Q_\e \to \rho_0$ in $C^{2,\delta} (\overline{\Omega})$ with $\delta > 0$ for integer $k \geq 4$ as required for the Schr\"odinger dynamics below and  $Q_\e \to \rho_0$ in $C^{1,\delta} (\overline{\Omega}) $, $\delta > 0$ for integer $k \geq 3$ as required for the gradient flow dynamics.\footnote{Note, by interpolation arguments, we can actually run the convergence argument is $H^{3+} (\Omega)$, where by $\cdot +$ we mean $\cdot + \nu$ for any $\nu > 0$.  We chose to here work with integer Sobolev spaces} for convenience.  In fact, the sharp estimate would include Schauder theory estimates directly on H\"older norms, but that would require uniform convergence in $C^1$, which does not seem obvious given that the $\nabla Q$ term appears to vanish in the limit preventing the proof of uniform bounds of higher regularity.
\end{prop}

\begin{rem}
In the appendix, we will actually decompose $\eta_\e = 1 +\frac{\tQ_\e}{|\log\e|}$, in which case
\[ Q_0 =  \lim_{\e \to 0} \left[ 2 \tQ_\e + \frac{\tQ_\e^2}{\logep}  \right]  ;\]
this convention slightly simplifies the resulting elliptic analysis.
\end{rem}

\begin{rem}
Similar results hold assuming Dirichlet boundary conditions, but we work here with Neumann as they are the most physically relevant.
\end{rem}


\section{Excess Energy Control}

\label{s:excess}
In this section we present an excess energy identity similar to \cite{CJ, JSp2}.  The
influence of the background requires some modifications, so we give a full proof. The form of the estimate will be similar to that found in \cite{BJS}.  

Recall from \eqref{eq:rho} the quantity
$$
r_\alpha = {1\over 8} \min \LCB \min_{i\neq j}  |\alpha_i - \alpha_j|, \min_i \dist (\alpha_i, \p \Omega)  \RCB.
$$
Then, we have  
\begin{prop} \label{p:excessenergy}
 Let $w \in H^1(\Omega; \C)$  and assume that there exists $\alpha$ such that $D_\e(\alpha, w) \leq 1$. Then there exists constants $\e_0 > 0$ and $\mu_0 > 0$ 
such that if $\e<\e_0$,  $r \leq r_\alpha$, 
and   
$$
\mu \equiv \LN J(w) -  \pi \sum d_j \delta_{\alpha_j}  \RN_{\dot{W}^{-1,1}} \leq \mu_0
$$ 
 then 
\begin{equation} \label{e:excessenergy}
\int_{\Omega_r(\alpha)} e^{\etae}_\e (|w|) + {1\over 8} \LV { j(w) \over |w|} - j(w_*) \RV^2 dx 
\leq D_\e + 
o_{\e, \mu}(1)
\end{equation}
where $o_{\mu,\e}(1)$ vanishes as $\e,\mu \to 0$.  
We also have 
\begin{align} 
\label{e:potbound}
\int_{\Omega} e_\e^1(|w|) dx & \leq C,  \\
\label{e:Lpbound}
\LN j(w) \RN_{L^p(\Omega)} & \leq C, \qquad 1 \leq p < 2.
\end{align}
Furthermore, there exists points $\xi_j \in B_{r/2}(\alpha_j)$ such that 
\begin{equation}  \label{e:localization}
\LN J(w) - \pi \sum d_j \delta_{\xi_j}  \RN_{\dot{W}^{-1,1}} \leq C \e \logep 
\end{equation}
and
\begin{equation}  \label{e:equipart}
\LN {\LC \p_k w, \p_\ell w\RC \over \logep} -  \delta_{k \ell} \sum \pi  \delta_{\xi_j}  \RN_{\dot{W}^{-1,1}} \leq C \logep^{-{1\over2}} ,
\end{equation}
where $\delta_{k \ell}$ is the Kronecker delta.  Here $C$ is a constant depending on $Q_0$, $\Omega$, $r_\alpha$, $n$, $p$, $\mu_0$, and $\e_0$.  

\end{prop}
\begin{rem} \label{r:Lqmod}
Note that \eqref{e:potbound} and Rellich-Kondrachov implies 
\begin{equation} \label{e:Lqmodstrongconv}
\LN |w| - 1 \RN_{L^q(\Omega)} = o_{\e, \mu}
\end{equation}  for all $q < +\infty$. 
\end{rem}

\begin{proof}
Although it is possible to get explicit control on the error $o_\e(1)$ in \eqref{e:excessenergy} by a careful analysis
as in \cite{JSp2, KS2}, the weaker estimate \eqref{e:excessenergy} is sufficient to prove the vortex motion law.

1.  We first decompose the excess energy into  
\begin{align*}
D_\e & = E^{\etae}_\e(w) - H_\e(\alpha, Q_0) \\ 
& = \int_{\Omega} e_\e^{\etae}(w) dx  - \LB \sum_{j=1}^n \LC \pi \logep + \gamma_0 + \pi Q_0(\alpha_j) \RC + W(\alpha) \RB \\
& = \int_{\Omega_\sigma(\alpha)} e_\e^{\etae}(w) - e_\e^{\etae} (w_*) dx  +  \int_{\Omega_\sigma(\alpha)}  e_\e^{\etae} (w_*) dx - W(\alpha)   \\
& \quad 
+  \sum_{j=1}^n \LB \int_{B_\sigma(\alpha_j)}  e_\e^{\etae} (w) dx - \LC \pi \logep + \gamma_0  +  \pi Q_0(\alpha_j)\RC  \RB \\
& = I + II +III, 
\end{align*}
where $\sigma$ is appropriately chosen and $\sigma \leq r \leq {r_\alpha} $.  

We will control terms $I - III$ by variants of  estimates found in \cite{JSp2, KS2, JSm}. 
  We will choose 
\begin{equation} \label{e:sigmadef}
\sigma = \max\{ \logep^{-{1\over 3}}, \sqrt{ \mu} \}
\end{equation}
 in the following.

2. By a simple variation of a standard calculation, see for example \cite{JSp2},
\begin{equation} \label{e:lassmiron}
{\etae^2 \over 2} \LV {j(w) \over |w|} - j(w_*) \RV^2 + e_\e^{\etae} (|w|) 
= e_\e^{\etae} (w) - e_\e^{\etae}(w_*) + \etae^2 j(w_*) \cdot \LC j(w_*) - {j(w)\over |w|} \RC, 
\end{equation}
and so our primary concern will be with estimating the integral of \eqref{e:lassmiron} on a suitable subdomain $\Omega_\sigma(\alpha)$.  
This last term can be rewritten as
\begin{align*}
\etae^2 j(w_*) \cdot \LC j(w_*) - {j(w)\over |w|} \RC
& = \LC {\etae^2 } -1 \RC j(w_*) \cdot \LC j(w_*)  - {j(w)\over  |w|} \RC \\
& \quad + j(w_*) \cdot \LC j(w_*) - {j(w) \over |w|}  \RC  .
\end{align*}

We now fix our $\mu_0$ and $\e_0$ in order to be able to cite some estimates from \cite{JSp1, JSp2}.
The tool we wish to use is in the proof of Lemma $4$ in \cite{JSp2}, which yields an estimate given some assumptions on the Jacobian 
and $\e$.  Let $K_1$ be the constant in the assumptions of Lemma $4$ from \cite{JSp2}, which depends only on $\Omega$.   

We follow an argument in the proof of Theorem 6.1 of \cite{BJS} for fixing these constants; however, there
are several new constraints including the fact that the vortex ball radius, $\sigma$, cannot be too small due to the need to control terms like $\LN \eta_\e^2 - 1 \RN_{L^\infty} \LN j(w_*) \RN_{L^\infty(\Omega_\sigma)}$. 

We first choose $\mu_0$.  Set $\mu_1$ such that 
\begin{equation} \label{e:eta0}
4 \mu \leq \sqrt{ \mu} \leq  \min \LCB r, {r_\alpha \over n K_1} \RCB
\end{equation}
for all $0 \leq \mu \leq \mu_1$, i.e. $\mu_1 = \min\LCB {1 \over 16} , {r^2 }, {r^2_\alpha \over n^2   K^2_1} \RCB$.
We further restrict $\mu_0 = \min\{ \mu_1, {r_\alpha \over 8 K_2 n^5} \}$ where $K_2$ comes from the assumptions in
Theorem 3 of \cite{JSp2}, needed for the localization result \eqref{e:localization}.

We now choose an $\e_0$ to allow us to use the necessary array of estimates.
 First, we choose $\e_2$ so that for all $\e < \e_2$,  
\begin{equation} \label{e:e2}
\logep^{-{1\over 3}} \leq \min \LCB r, {r_\alpha \over n K_1} \RCB.
\end{equation}
We then set $\e_1 \leq \e_2$ so that $\e E^{1}_\e(w) \leq \sqrt{\e}$ for all $\e < \e_1$.  
Indeed  since 
$\eta^2_\e \geq {1\over2}$ then $E^1_\e(w) \leq {1\over \min \eta_\e^4} E_\e^{\etae}(w) \leq 4 C \logep$, so there exists
an $\e_1$ such that $\e E^1_\e(w) \leq 4 C \e  \logep \leq \sqrt{\e}$ for all $\e \leq \e_1$.   
%
We next fix $\e_0 \leq \e_1$ such that, 
\begin{equation} \label{e:e0}
\e \log {r_\alpha \over \e} \leq \logep^{-{2\over 3}}
\end{equation}
for all $0 \leq \e \leq \e_0$.


If $ \mu \leq  \logep^{-{2\over 3}}$ then we use $s_\e = \logep^{-{2\over 3}}$ 
and $\sigma =  \logep^{-{1\over 3}} $ in
Lemma 4 of \cite{JSp2}.  In particular \eqref{e:e2} and \eqref{e:e0} imply that the assumptions in Lemma 4 hold, and we find
\begin{equation}
\begin{split}
& \LV \int_{\Omega_{\sigma} (\alpha)}  j(w_*) \cdot \LC j(w_*) - {j(w) \over |w|}  \RC dx \RV  \\
& \leq  {1\over 4} \int_{\Omega_\sigma(\alpha)} \LV {j(w) \over |w|} - j(w_*) \RV^2 dx  
 +  C  \logep^{{1\over 3}} \LC \logep^{-{2\over 3}} + \sqrt{\e} \RC +  {C \over r_\alpha} \logep^{-{1\over3}} \\
 & \leq  {1\over 4} \int_{\Omega_\sigma(\alpha)} \LV {j(w) \over |w|} - j(w_*) \RV^2 dx  
 +  o_{\e, \mu}(1).
\end{split}
\end{equation}

If $\logep^{-{2\over3}} \leq \mu$ then again the assumptions hold for Lemma 4 of \cite{JSp2}, and 
choose $s_\e = \mu$ and $\sigma = \sqrt{ \mu }$ so    
\begin{equation}
\begin{split}
& \LV \int_{\Omega_{\sigma} (\alpha)}  j(w_*) \cdot \LC j(w_*) - {j(w) \over |w|}  \RC dx  \RV \\
& \leq  {1\over 4} \int_{\Omega_\sigma(\alpha)} \LV {j(w) \over |w|} - j(w_*) \RV^2 dx   +  {C \over \sqrt{\mu}} \LC \mu + \sqrt{\e} \RC 
+ C{ \mu \over r_\alpha}  \\
 & \leq  {1\over 4} \int_{\Omega_\sigma(\alpha)} \LV {j(w) \over |w|} - j(w_*) \RV^2 dx  
 +  o_{\e, \mu}(1).
\end{split}
\end{equation}

Taking the sum over the two errors yields sufficient control on $\int_{\Omega_\sigma} j(w_*) \cdot \LC j(w_*) - {j(w) \over |w|}  \RC dx$.

Finally, we note 
\begin{align*}
& \LV \int_{\Omega_\sigma(\alpha)}  \LC {\etae^2 } -1 \RC j(w_*) \cdot \LC j(w_*)  - {j(w)\over  |w|} \RC   dx \RV \\ 
& \qquad \leq C \LN 1 - \eta^2_\e \RN_{L^\infty} \LN j(w_*) \RN_{L^\infty (\Omega_\sigma(\alpha)) }\LC 1 +
  \int_{\Omega_\sigma(\alpha)}   \LV j(w_*)  - {j(w)\over  |w|} \RV^2 dx \RC \\
  & \qquad \leq C {\min\{ \logep^{1\over3},  \sqrt{1 \over \mu } \}  \over \logep }  \LC 1 +  \int_{\Omega_\sigma(\alpha)}   \LV j(w_*)  - {j(w)\over  |w|} \RV^2 dx \RC \\
  &\qquad \leq o_{\e, \mu} (1) \LC 1 +  \int_{\Omega_\sigma(\alpha)}   \LV j(w_*)  - {j(w)\over  |w|} \RV^2 dx \RC.
\end{align*}
Combining these estimates together controls term $I$.

3.  We control  $II$ using the explicit control in Lemma 12 \cite{JSp2},  
\begin{align*}
& \LV \int_{\Omega_\sigma(\alpha )}  e_\e^{\etae} (w_*) dx - \LC  W(\alpha) + n \pi \log {1\over \sigma} \RC \RV \\
& = \LV {1\over 2} \int_{\Omega_\sigma(\alpha )} \eta_\e^2 \LV \nabla w_* \RV^2 dx - \LC  W(\alpha) + n \pi \log {1\over \sigma} \RC \RV \\
& \leq \LV  {1\over 2} \int_{\Omega_\sigma(\alpha )}  \LV \nabla w_* \RV^2 dx - \LC  W(\alpha) + n \pi \log {1\over \sigma} \RC \RV 
 \quad + C \LN 1 - \eta^2_\e \RN_{L^\infty} \LN j(w_*) \RN^2_{L^\infty(\Omega_\sigma)} \\
& \leq C {\sigma^2 \over r_\alpha^2} + C { 1 \over \sigma^2 \logep}
 \leq C   {\max\{ \logep^{-{2\over 3}}, { \mu } \} \over r_\alpha^2} +  C {\min\{ \logep^{{2\over 3}}, {1 \over  \mu } \}\over \logep } 
= o_{\e, \mu}(1).
\end{align*}

4.  We now consider terms $III$.  
Recall that $Q_\e(x) = \logep \LC \eta_\e^2(x) - 1 \RC$.   
We set $q_\e(r,x_0)=\inf_{x\in B_r( x_0)} Q_\e(x)$ and $\widetilde \e_r = {\e \over 1 + q_\e(r,x_0)}$.  Recall from \eqref{e:e2} that
\begin{align*}
\LN J(w) - \pi d_j \delta_{\alpha_j} \RN_{\dot{W}^{-1,1}(B_\sigma(\alpha_j))}
 \leq \LN J(w) - \pi \sum d_k \delta_{\alpha_k} \RN_{\dot{W}^{-1,1}(\Omega)} 
 \leq \mu \leq {\sqrt{\mu} \over 4} \leq {\sigma \over 4};
\end{align*}
therefore, we can invoke Theorem~$1.3$ and Lemma~$6.8$ of \cite{JSp1}: 
\begin{align*}
\int_{B_\sigma(\alpha_j)} e_{\widetilde{\e}_\sigma}^1(w) dx & \geq \pi \log {\sigma \over \widetilde{\e}_\sigma} 
+ \gamma_0  - C {\widetilde{\e}_\sigma \over \sigma} \sqrt{\log { \sigma \over \widetilde{\e}_\sigma}}
- C {\mu \over \sigma} \\
& \geq  \pi \log {\sigma \over \e}  + \gamma_0 + \pi \log (1 + {Q_\e(\sigma, \alpha_j) \over \logep}) - C { \e \logep }  
- C \sqrt{\mu } ,
\end{align*}
so
\begin{equation}
\int_{B_\sigma(\alpha_j)} e_{\widetilde{\e}_\sigma}^1(w) dx \geq
\pi \log {\sigma \over \e}  + \gamma_0 - o_{\e,\mu}(1).
\end{equation}
Therefore,
\begin{align*}
& \int_{B_\sigma(\alpha_j)} e^{\etae}_\e(w) dx - \LC \pi \log {\sigma \over \e} + \gamma_0 +  \pi Q_0(\alpha) \RC  \\
& \ge \LC 1+ {Q_\e(\sigma,\alpha_j) \over \logep} \RC  \LC \pi \log {\sigma \over \e}  + \gamma_0 - o_{\e,\eta}(1) \RC  
 - \LC \pi \log {\sigma \over \e} + \gamma_0 +  \pi Q_0(\alpha) \RC  \\
 & \geq -\pi  |q_\e(\sigma, \alpha_j) - Q_0(\alpha_j)| - C { \log \sigma + 1 \over \logep} \\
 &  \geq -  \pi |q_\e(\sigma, \alpha_j) - Q_0(\alpha_j)| -C {  \log \logep  \over \logep} 
  = - o_{\e, \mu}(1),
\end{align*}
 and so \eqref{e:excessenergy} follows.


5.  We next turn to the proof of \eqref{e:potbound} and \eqref{e:Lpbound}. 
 We first choose vortex balls $B_{r_\alpha}(\alpha_j)$; by \eqref{e:eta0}, 
\begin{equation}
\label{e:Jacwlocralpha}
\LN J(w) - \pi  d_j \delta_{\alpha_j} \RN_{\dot{W^{-1,1}(B_r(\alpha_j))}} 
\leq \LN J(w) - \pi \sum d_k \delta_{\alpha_j} \RN_{\dot{W^{-1,1}(\Omega)}}  \leq \mu \leq {r_\alpha \over 4}, 
\end{equation} 
and arguing as in Step 4,
\begin{equation*} \label{e:singlevortexlower}
\int_{B_{r_\alpha}(\alpha_j)} e_\e^{\etae} (w)dx \geq \pi \logep -   C(r_\alpha, Q_0).
\end{equation*}
Therefore,
\begin{equation} \label{e:cutoutbound}
\begin{split}
\int_{\Omega_{r_\alpha} (\alpha)} e_\e^{\etae}(w)dx & = E_\e^{\etae}(w) - \sum_{j=1}^n e_\e^{\etae}(w) dx  \\
& \leq D_\e + n \gamma_0 + \pi \sum_{j=1}^n Q_0(\alpha_j) + W(\alpha) +  C(n, r_\alpha, Q_0) ,
\end{split}
\end{equation}
and so,
\begin{align}
\int_{B_{r_\alpha}(\alpha_k)} e_\e^{\etae}(w) dx & = E_\e^{\etae}(w) -  \int_{\Omega_r(\alpha)} e_\e^{\etae}(w) dx 
- \sum_{j\neq k} \int_{B_{r_\alpha}(\alpha_j)} e_\e^{\etae}(w) dx  \nonumber \\
& \leq \pi \logep +  C(n, r_\alpha, Q_0)  \label{e:singlevortexenergyrQ} . 
\end{align}
Finally, $\etae^2 \geq{1\over 2}$, \eqref{e:singlevortexenergyrQ}, and \eqref{e:Jacwlocralpha} allow us to use (4.27) in the proof of Proposition  4.2 in \cite{JSp1} which in turn implies
\begin{equation} \label{e:modboundball}
\int_{B_r(\alpha_j)} e_\e^1(|w|) dx \leq  C(n, r_\alpha, Q_0). 
\end{equation}
 Since $r \leq r_\alpha$, we can combine \eqref{e:excessenergy}, $\eta_\e^2 \geq {1\over 2}$, and \eqref{e:modboundball} to get bound \eqref{e:potbound}.

To prove \eqref{e:Lpbound} we use an $L^p$ bound on each vortex ball $B_{r_\alpha}(\alpha_j)$ in Theorem~3.2.1 of \cite{CJ}.  
Outside the vortex balls, we use use the decomposition argument; in particular,  
since $\eta^2_\e(x) \geq {1\over 2}$ then 
${1 \over 4} \int_{\Omega_{r_\alpha}(\alpha)} \LV { j( w)  \over |w|} \RV^2 dx  \leq \int_{\Omega_{r_\alpha}(\alpha)} e_\e^{\eta_\e}(w) dx \leq C(D_\e, r_\alpha)$ from \eqref{e:cutoutbound}.  
From Remark~\ref{r:Lqmod} and the above bound, one finds
\begin{align*}
\LN j(w)  \RN_{L^p(\Omega_{r_\alpha}(\alpha))} & \leq \LN {j(w) \over |w|} \RN_{L^p(\Omega_{r_\alpha}(\alpha))} + \LN {j(w) \over |w|} (1 - |w|  )\RN_{L^p(\Omega_{r_\alpha}(\alpha))} \\
& \leq \LN {j(w) \over |w|} \RN_{L^2(\Omega_{r_\alpha}(\alpha))} \LC C + \LN 1 - |w| \RN_{L^{2 p \over 2 - p}(\Omega_{r_\alpha}(\alpha))} \RC \leq C(r_\alpha, Q_0, p), 
\end{align*}
due to Remark~\ref{r:Lqmod}.  Combining the vortex ball bound with the bound in the excised domain yields \eqref{e:Lpbound}.

5.  The restriction on $\mu_0$ allows us to use Theorem 3 of \cite{JSp2}, and estimate \eqref{e:localization} follows directly.  To establish 
\eqref{e:equipart} we can follow the proof of Theorem 11 in \cite{KS2}, along with the equipartioning result of \cite{KSEquipartition}, which 
lets us localize the stress-energy tensor $\nabla w \otimes \nabla w$ up to an error of order $O(\sqrt{\logep})$.

\end{proof}


\section{Proof of Theorem~\ref{t:vortexdynamics}}

The proof of the vortex motion law entails first establishing the smooth evolution of the vortex paths $\xi_j(t)$ via the 
compactness theorems in Section~\ref{s:excess}.  The second step will compare the vortex paths with the ODE \eqref{e:ODE}.

\subsection{Lipschitz Continuity of the Vortex Paths}

\begin{prop} \label{p:lipschitzpaths}

Let $\{a_{j}^0, d_j\}$ be a configuration of vortices such that $d_j \in \{-1, 1\}$. 
Let $w_\e, \ 0 < \e <1$, satisfy \eqref{e:weightedGP} under well-preparedness condition \eqref{e:wellpreparedness}. There exist a time $T > 0$, 
a sequence $\e_k \to 0$, and $\xi_j \in Lip([0,T_0], \Omega)$  such that $\xi_j(0) = a^0_j$ 
and 
\begin{equation} \label{e:timelocalization}
\sup_k \LN J(w_{\e_k}(\cdot , t ) ) - \pi \sum d_j \delta_{\xi_j(t)} \RN_{\dot{W}^{-1,1}(\Omega)} \to 0 
\end{equation}
 as $k \to \infty$ and $t \in [0,T_0]$.  

\end{prop}

\begin{proof}
1.  We first claim that for any $T > 0$, there exists an $\e$ such that $J(w_{\e})$ is a continuous function on $[0,T]$ into $L^1(\Omega)$.  
Since $J(u) = - \nabla^\perp \mathcal{R} u \cdot \nabla \mathcal{I} u$ then
\[
\LN J(w_\e(t)) - J(w_\e(s)) \RN_{L^1}
\leq C \LN \nabla w_\e (t) - \nabla w_\e (s) \RN_{L^2} \LN \nabla w_\e (t) + \nabla w_\e (s) \RN_{L^2} \leq C_\e o_{|t-s|}(1),
\]
since the solution map is a continuous function into  $H^1(\Omega)$. 

2.   We use \eqref{e:modJacev2} along with the results in Proposition~\ref{p:excessenergy}.  Set 
\begin{equation} \label{e:goodtestfcn}
\varphi_j(x) = x \chi \LC {|x - a^0_j | \over r_{a^0}} \RC
\end{equation}
where 
\[
\chi(s) = \left\{ \begin{array}{ll}  1 & \hbox{ for }  s \leq 1 \\  0 & \hbox{ for } s > 2 \end{array} \right.  .
\]
Choosing $0 \leq  r \leq s$ and using  \eqref{e:modJacev2} and \eqref{e:wellpreparedness} we generate the following bound, 
\begin{align*}
 \LV \int_r^s \int_\Omega \varphi_j(x) \p_t J(w_\e (\cdot, t))  dx dt \RV
& \leq C  \LN \nabla^2 \varphi_j \RN_{L^\infty} \int_r^s \int_{\Omega_{r_a} (r) }  \LV \nabla w_\e (t) \RV^2 dx  \\
& \quad + C {\LN \nabla \varphi_j \RN_{L^\infty}  \over \logep} \int_r^s \int_{\Omega} e_\e^{\etae}(w_\e) dx ds  \\ 
& \leq C |s -r |, 
\end{align*}
and so
\begin{align*}
\pi \LV \xi_j(s) - \xi_j(r) \RV & =
\pi \LV  \int  \varphi_j(x) \LC   \delta_{\xi_j(s)} -  \delta_{\xi_j(r)}  \RC dx \RV  \\
&  \leq C  \LN \nabla \varphi_j \RN_{L^\infty}  \sup_{t\in[r,s]} \LN J(w_\e(\cdot, t)) - \pi \sum d_j \delta_{\xi_j(t)}  \RN_{\dot{W}^{-1,1}}  \\
& \qquad \quad  +
\LV \int_r^s \int_\Omega \varphi_j(x) \p_t J(w_\e (\cdot, t))  dx dt \RV \\
& \stackrel{\eqref{e:localization}}{ \leq} C | s - r |+ o_\e(1).
\end{align*}

We can now employ a diagonalization argument.  Taking subsequence $\e_k \subset \e$, we generate 
a dense, countable collection of times $t \in [0,T]$, in which $T$ depends solely on $r_a$, such that the 
$\xi_j(t)$ satisfy $\xi_j(0) = a_j^0$ with \eqref{e:timelocalization} as $k \to \infty$.  Since the collection is dense, we can take the limit as $k \to \infty$.  The Lipschitz extension of the $\xi_j(t)$'s on $[0,T]$ satisfy the same conditions.
\end{proof}


\subsection{Vortex Dynamics}
In this subsection we complete the proof of Theorem~\ref{t:vortexdynamics}.  We use the Lipschitz paths $\xi_j(t)$ 
generated along the subsequence $\e_k$ from Proposition~\ref{p:lipschitzpaths}.  We will examine the differences between vortex paths and the ODE, which will lead eventually to a Gronwall argument.  

The proof requires two technical calculations. The first 
describes the gradient of $W(\alpha)$ in terms of the canonical harmonic map $w_*$.  

\begin{lem}
\label{l:vortdyn1}
Let $\alpha \in \Omega$ then $w_* = w_*(\cdot, \alpha)$ and the renormalized 
energy $W(\alpha)$ satisfy
\begin{equation}  \label{e:gradRenorm}
\int \J_{k \ell} \p_{k} \p_{m} \varphi j(w_*(\cdot, \alpha))_m j(w_*(\cdot, \alpha))_\ell  dx
= -\sum d_j \J_{\ell k} \p_k \varphi (\alpha_j) \LC \nabla_{\alpha_j} W(\alpha) \RC_\ell
\end{equation}
where $\varphi \in C^2(\Omega)$ and $\nabla^2 \varphi$ has support in a neighborhood 
of the $\alpha_j$'s. 

\end{lem}
\begin{proof}
See for example \cite{CJ,LX} for a proof.
\end{proof}

The second technical result proves  weak compactness of the supercurrent away from the vortex cores.  This 
will be used to bound certain cross terms in the Gronwall argument.  The argument is along the lines of similar 
proofs in \cite{CJ, LX}, among other places, and we sketch most of the argument.  


\begin{lem}  \label{l:weakL2conv} Let $w_\e$ be a solution to \eqref{e:weightedGP} under the hypotheses of Proposition~\ref{p:lipschitzpaths} 
and suppose $X \in C_0^\infty(\Omega \times [0,T])$ is a smooth vector field with support away from the vortex paths.  For any $0 < t_0 < T$ and $\tau > 0$ with
$t_0 + \tau < T$, then 
 \begin{equation} \label{eq:distconv}
\int_{t_0}^{t_0 + \tau} \int_\Omega X \cdot \LC { j(w_{\e}(t)) \over |w_{\e}(t)|} - j(w_*(\xi(t))) \RC dx dt \to 0
\end{equation}
as $\e \to 0$, 
where $\xi(t) = (\xi_1(t), \ldots, \xi_n(t))$ are the Lipschitz vortex paths generated in Proposition~\ref{p:lipschitzpaths}.
\end{lem}

\begin{proof}

We sketch the argument in \cite{KMMS} for simplicity.   Since 
$$ { j(w_{\e}(t)) \over |w_{\e}(t)|} - j(w_*(\xi(t))) = { j(w_{\e}(t)) \over |w_{\e}(t)|} \LC 1 - |w_\e(t)| \RC +  { j(w_{\e}(t))} - j(w_*(\xi(t))) ,$$ 
then we can control 
\begin{align*}
& \int_{t_0}^{t_0 + \tau} \int_\Omega X \cdot \LC { j(w_{\e}(t)) \over |w_{\e}(t)|} - j(w_*(\xi(t))) \RC dx dt \\
&\leq  \int_{t_0}^{t_0 + \tau} \int_\Omega X \cdot  { j(w_{\e}(t)) \over |w_{\e}(t)|} \LC 1 - |w_\e(t)| \RC dx dt \\
& \quad + \int_{t_0}^{t_0 + \tau} \int_\Omega X \cdot \LC  j(w_{\e}(t)) - j(w_*(\xi(t))) \RC dx dt \\
&\leq C \e \logep \LN X \RN_{L^2(\Omega \times [0,T])}  + \int_{t_0}^{t_0 + \tau} \int_\Omega X \cdot \LC { j(w_{\e}(t)) } - j(w_*(\xi(t))) \RC dx dt.
\end{align*}
We now focus on the last integral. 

By \eqref{e:Lpbound} in Proposition~\ref{p:excessenergy} the $j(w_\e)$ are uniformly bounded in $L^p$ for all $1 \leq p < 2$; furthermore, the  $j(w_*(\xi(t)))$ are also in $L^p$ for all $1 \leq p <2$.    The uniform bounds allows us generate a Hodge decomposition, 
$j(w_\e) - j(w_*(\xi(t))) = \nabla h^\e_1(t) + \mathbb{J} \nabla h^\e_2(t)$ where the $h^\e_j (t) \in W^{1,p}(\Omega)$ and 
\begin{align*}
\Delta h^\e_1(t) = \dv j(w_\e(t)) \hbox{ in } \Omega \quad &  \hbox{ and } \quad  \p_\nu h^\e_1(t) = 0 \hbox{ on } \p \Omega, \\
\Delta h^\e_2(t) = 2 (J(w_\e(t)) - J(w_*(\xi(t)))) \hbox{ in } \Omega \quad &  \hbox{ and } \quad  h^\e_2(t) = 0 \hbox{ on } \p \Omega. 
\end{align*}

We  first claim that 
\begin{equation} \label{eq:h1est}
\LV \int_{t_0}^{t_0 + \tau} \int_\Omega X \cdot \nabla h^\e_1 dx dt \RV \to 0 \hbox{ as } \e \to 0. 
\end{equation} 
Since the vector $X(t) \in C_0^\infty(\Omega)$ at each time then we can define the unique Hodge decomposition $X = \nabla \phi + \mathbb{J}\nabla \psi$  such that 
\begin{align*}
\Delta \phi(t) = \dv X \hbox{ in } \Omega \quad &  \hbox{ and } \quad  \p_\nu \phi(t) = 0 \hbox{ on } \p \Omega, \\
\Delta \psi(t) = \curl X \hbox{ in } \Omega \quad &  \hbox{ and } \quad   \psi(t) = 0 \hbox{ on } \p \Omega. 
\end{align*}
We set $\overline \phi$ to be the average of $\phi$ over $\Omega$; and consequently from the Neumann boundary conditions on $w_\e$,  
\begin{align*}
& \LV \int_0^T \int_\Omega   X \cdot \nabla h_1^\e dx dt \RV 
=  \LV \int_0^T \int_\Omega   \LC \nabla \phi + \mathbb{J} \nabla \psi \RC \cdot \nabla h_1^\e dx dt \RV \\
 &  \hspace{1cm} =  \LV \int_0^T \int_\Omega  (  \phi  -  \overline \phi) \dv j(w_\e(t)) dx dt \RV \\
& \hspace{1.2cm}  \leq \LV \int_0^T \int_\Omega  ( \phi -  \overline \phi) {\nabla \eta_\e^2 \cdot j(w_\e) \over \eta_\e^2}  dx dt \RV +
 \LV \int_0^T \int_\Omega  {\LV \p_t ( \phi -  \overline \phi) \RV \over 2 \eta_\e^2}   \LV 1 - |w_\e|^2 \RV  dx dt \RV \\
 & \hspace{2cm}  \leq { C \over \logep} \LN  X \RN_{L^4(\Omega \times [0,T])} + C \e \LN \p_t  X \RN_{L^2(\Omega \times [0,T])} 
  = o_\e(1)   ,
\end{align*}
where we used a Hodge decomposition of $\p_t X$ that arises from differentiating $\phi$ and $\psi$ in time.   
In order to control the term with $h_2^\e$, we use elliptic estimates and bound 
$$\LN \nabla h_2^\e (t) \RN_{L^p}  \leq  C \LN J(w_\e(t)) - J(w_*(t)) \RN_{W^{-1,p}} \leq C \LN J(w_\e(t)) - J(w_*(t)) \RN_{W^{-1,1}} = o_\e(1).$$  
This implies  $\LV \int_{t_0}^{t_0 + \tau} \int_\Omega X \cdot \mathbb{J} \nabla h_2^\e dx dt \RV = o_\e(1)$, and
the three bounds combine together to yield \eqref{eq:distconv}.

\end{proof}

With Lemmas \ref{l:vortdyn1} and \ref{l:weakL2conv} in hand, we can now turn to the proof of the main result.

\begin{proof}[Proof of Theorem~\ref{t:vortexdynamics}]

Define 
\[
\mu(t) \equiv \sum_j \LV \alpha_j(t) - \xi_j(t) \RV = \sum_j \mu_j(t).
\]
We will estimate the difference between the vortex paths $\xi_j$, generated by \eqref{e:weightedGP}, and
the positions $a_j$, extracted by the ODE \eqref{e:ODE}.
We also let $D_{\e}(w_\e(t), b(t)) = E_\e^{\etae}(w_\e(t)) - W_\e(b(t))$ denote the time-varying excess energy. 

Recall from \eqref{e:ODE} that we will estimate the differences between vortex paths and the ODE positions:
\begin{align}
\pi \dot{\alpha}_j 
\nonumber & = - \nabla_{\alpha_j}^\perp W(\alpha) + \pi \nabla^\perp Q_0(\alpha_j) \\
\label{e:adjustODE} 
& = - \nabla_{\xi_j}^\perp W(\xi) + \pi \nabla^\perp Q_0(\xi_j) + \nabla^\perp_{\xi_j} W(\xi) - \nabla^\perp_{\alpha_j}W(\alpha) + \pi \nabla^\perp Q_0(\alpha_j) - \pi \nabla^\perp Q_0(\xi_j).
\end{align}
Choosing a time interval $0 \leq a < b \leq T$ and using  \eqref{e:modJacev2},  \eqref{e:goodtestfcn}, \eqref{e:gradRenorm}, \eqref{e:adjustODE},  
\begin{align}
\nonumber 
\pi \LC   \mu_j(b) - \mu_j(a) \RC 
& \leq \int_a^b \LV \nabla_{\xi_j} W(\xi(t)) - \nabla_{\alpha_j}W(\alpha(t))\RV dt  + \pi \int_a^b \LV \nabla  Q_0(\xi_j(t)) - \nabla Q_0(\alpha_j(t)) \RV dt \\
\nonumber 
& \quad +  \limsup_{k \to \infty} \LV \int_a^b  \int_{B_{r}(\alpha_j)} \J_{\ell p } \p_p \p_q \varphi 
\LB \LC \p_q w_{\e_k} , \p_\ell w_{\e_k} \RC - \LC \p_q w_*(\xi) , \p_\ell w_*(\xi) \RC \RB  dx dt \RV  \\
\nonumber 
& \quad + \limsup_{k \to \infty} \LV \int_a^b  \LB  \int_{B_{r}(\alpha_j)} \J_{\ell p } \p_p  \varphi { \p_q \eta_{\e_k}^2 \over \eta_{\e_k}^2} \LC \p_\ell w_{\e_k} , \p_q w_{\e_k} \RC dx 
- \pi \nabla^\perp Q_0(\xi_j) \RB dt \RV \\
\nonumber 
& \quad + \limsup_{k \to \infty} \LV  \int_a^b  \int_{B_{r}(\alpha_j)} \J_{\ell p } \p_p  \varphi \p_\ell \eta_{\e_k}^2 { \LC 1 - |w_{\e_k}|^2 \RC^2 \over 4 \e^2 } dx  dt  \RV \\
\label{e:listmu}
& = \mathcal{A} +\mathcal{B} +\mathcal{C} +\mathcal{D} +\mathcal{E}. 
\end{align}

We can control $\mathcal{A}$ and $\mathcal{B}$ using our Lipschitz bounds; in particular, 
\begin{align} 
\nonumber \int_a^b |\nabla W(\xi_j(t)) - \nabla W(\alpha_j(t)) |  dt
&  \leq \int_a^b \sup_{j: |\xi_j - \alpha_j| \leq r_a/4} | \nabla^2 W (\alpha) |    \LC \sum |\xi_j(t) - \alpha_j(t)| \RC dt  \\
\label{e:Abound}
& \leq {C } \int_a^b \mu(t) dt,
\end{align}
where we used Lemma~10 of \cite{JSp2} to control the Hessian of $W(a)$, with a bound dependent on $r_a$.  
Likewise, we can use Appendix~\ref{s:etacontrol} on the regularity of $\eta_{\e_k}$ to bound
\begin{align} \label{e:Bbound}
 \int_a^b |\nabla Q(\xi_j(t)) - \nabla Q(\alpha_j(t)) |  dt 
   \leq     \| Q_0 \|_{W^{2,\infty}} \int_a^b \mu(t) dt.  
\end{align}

We now turn to the control on $\mathcal{C}$.  For shorthand we define components $j_k = ({j(w_{\e_k}) \over |w_{\e_k}|})_k$ and $j^*_k = (j(w_*(\xi)))_k$ of the 
 supercurrents, and since $\p_k w_{\e_k} = (\p_k |w_{\e_k}| + i j_k ) { w_{\e_k} \over |w_{\e_k}|}$ then 
\begin{align*}
\mathcal{C} 
& \leq  \limsup_{k \to \infty} \LV   \int_a^b \int_{\Omega_{r}(\alpha)} \J_{\ell p } \p_p \p_q \varphi 
  \p_q |w_{\e_k}| \p_\ell |w_{\e_k}|    dx dt \RV \\
& \quad +  \limsup_{k \to \infty} \LV \int_a^b  \int_{\Omega_{r}(\alpha)} \J_{\ell p } \p_p \p_q \varphi 
 \LC j_q  - j_q^*, j_\ell - j^*_\ell  \RC   dx dt \RV \\
 & \quad +  \limsup_{k \to \infty} \LV  \int_a^b \int_{\Omega_{r}(\alpha)} \J_{\ell p } \p_p \p_q \varphi 
 \LC  j^*_q  \LC j_\ell - j^*_\ell  \RC \RC   dx dt \RV \\
 & \quad +  \limsup_{k \to \infty} \LV  \int_a^b \int_{\Omega_{r}(\alpha)} \J_{\ell p } \p_p \p_q \varphi 
 \LC \LC j_q - j^*_q \RC  j^*_\ell  \RC   dx  dt \RV \\
 & = \mathcal{C} _1 + \mathcal{C} _2 + \mathcal{C} _3 + \mathcal{C} _4
\end{align*}
and since $D_{\e_k}(w_{\e_k}, \xi) = D_{\e_k}(w_{\e_k}, \alpha) + W(\xi) - W(\alpha) + \pi Q_0(\xi) - \pi Q_0(\alpha)$, then
\begin{equation*}\label{e:a1a2}
\begin{split}
\mathcal{C} _1 + \mathcal{C} _2 & \stackrel{\eqref{e:excessenergy}}{\lesssim} \int_a^b D_{\e_k}(w_{\e_k}, \xi(t)) dt + o_{k}(1)  \\
& \stackrel{\eqref{e:Abound},\eqref{e:Bbound}}{\leq}  C \int_a^b \mu(t) dt + o_k(1),
\end{split}
\end{equation*}
since $D_{\e_k}(w_{\e_k}(t), \alpha(t)) = D_{\e_k}(w^0_{\e_k}, \alpha^0) = o_k(1)$.
 We set $X = \J_{\ell p } j^*_q \p_p \p_q \varphi$ or $\J_{\ell p } j^*_\ell \p_p \p_q \varphi$, both in $C_0^\infty(\Omega\times[0,T])$, then from Lemma~\ref{l:weakL2conv}, we have that $\lim_{k \to \infty} \mathcal{C} _3 + \mathcal{C} _4 = 0$; note this is where we needed to consider the time-integral on $(a,b)$.

Next we control $\mathcal{D}$.
\begin{align*}
\mathcal{D} & \leq \limsup_{k \to \infty} \LV  \int_a^b \int_{B_{r}(a_j)} \J_{\ell p } \p_p  \varphi { \p_q \eta_{\e_k}^2 \over \eta_{\e_k}^2} 
\LB \LC \p_\ell w_{\e_k} , \p_q w_{\e_k} \RC - \pi \LV \log \e_k \RV \delta_{\ell q} \sum \delta_{\xi_j(t)} \RB dx dt \RV  \\
& \quad + \limsup_{k \to \infty} \LV \int_a^b  \pi \LV \log \e_k \RV { \nabla^\perp \eta_{\e_k}^2(\xi_j(t)) \over \eta_{\e_k}^2(\xi_j(t))} 
     -     \pi \nabla^\perp Q_0(\xi_j(t) )  dt\RV \\
  & = \mathcal{D}_1 + \mathcal{D}_2 .
\end{align*} 
We can now bound
\begin{align*}
\mathcal{D}_1 & \leq T \limsup_{k \to \infty} \LB   \LN \nabla \LC  \p_p  \varphi { \p_q \eta_{\e_k}^2 \over \eta_{\e_k}^2} \RC \RN_{L^\infty}
\LN  \LC \p_\ell w_{\e_k} , \p_q w_{\e_k} \RC - \pi \LV \log \e_k \RV \delta_{\ell q } \sum \delta_{\xi_j} \RN_{\dot{W}^{-1,1}}  \RB \\
& \stackrel{\eqref{e:equipart}}{\lesssim}  \limsup_{k \to \infty}  {C \LV \log \e_k \RV^{-{1\over2}}\over r^2_\alpha  }  = o_{k}(1)
\end{align*}
and
\begin{align*}
\mathcal{D}_2 & \lesssim  \limsup_{k \to \infty} \int_a^b \LV      {    \p_p Q_{\e_k}(\xi_j) \over \eta_{\e_k}^2} 
       -       \p_p Q_0(\xi_j) \RV  dt = o_{k}(1) .
\end{align*}

Finally, we bound $\mathcal{E}$ using \eqref{e:potbound}:
\begin{align*}
\mathcal{E} \leq {C \over r_a \LV \log \e_k \RV} \int_a^b \int_{\Omega} { (1 - |w_{\e_k}(t)|^2 )^2 \over \e_k^2} dx dt \leq {C T\over r_a \LV \log \e_k \RV}
= o_k(1).
\end{align*}

Combining together estimates generates a differential inequality, 
\[
\mu(b) - \mu(a) \leq C \int_a^b \mu(t) dt,
\]
and choosing $b = a + \delta$ and dividing by $\delta$ implies the classical Gronwall inequality with $\mu(0) = 0$.  Therefore, $\mu(t) = 0$ for all $0 \leq t \leq T$.  We can then restart the problem at time $T$ and continue the dynamics until $r_a \to 0$.  

So far we have only shown the theorem for a subsequence. However, our argument can be applied to subsequences of any sequence of $\e_k\to 0$, 
and as the limit is independent of the chosen subsequence, we obtain convergence for $\e\to 0$ without taking subsequences. 
\end{proof}

\section{Second order Gamma Convergence and Gradient Flow}
\label{sec:gamma}
In this section we collect and restate some of our results on the static energy in the spirit of $\Gamma$-convergence, compare \cite{AP} for
the corresponding results for $p_\e\equiv 1$. For dynamics, we show that the approach of \cite{SSGamma} can be adapted to 
treat the gradient flow as opposed to  the Schr\"odinger dynamics studied in the previous sections.

All of our arguments work with either Neumann or Dirichlet boundary conditions, just as those in \cite{AP} and \cite{SSGamma}. For 
simplicity, we assume Neumann boundary conditions throughout, and sometimes comment on changes necessary for the Dirichlet case.

\subsection{Further $\Gamma$-convergence results}
We discuss $\Gamma$-convergence properties for $\Ee$. From this, an analysis for
 the original energy
\[
E_\e(u) = \int_\Omega e_\e(u) dx = \int_\Omega \frac12|\nabla u|^2 + \frac{1}{4\e^2} (p_\e^2-|u|^2)^2 dx
\]
can also be deduced by  the Lassoued-Mironescu result \cite{LM}
\begin{equation}\label{eq:LM1}
E_\e(w\eta_\e) = E_\e(\eta_\e)+E^{\eta_\e}_\e(w)
\end{equation}
that follows from the calculation
\[
e_\e(u) = e_\e(\eta_\e) + e_\e^{\eta_\e}(w) + {1\over 2} \dv \LC (|w|^2 - 1) \eta_\e \nabla \eta_\e \RC.
\]
If $|w|=|u|=\eta_\e=1$ on $\p\Omega$ or $\nu \cdot \nabla \eta_\e=0$ on $\p\Omega$, we can deduce \eqref{eq:LM1}
in the Dirichlet or Neumann case, respectively.
However,  $E_\e(\eta_\e)$ contributes very little to the energy, and all the dynamically interesting 
information is encoded in $\Ee(w)$, so we concentrate on $\Ee$ in the following.

\begin{prop}\label{prop:Gammac1}
Let $w_\e\in H^1(\Omega;\C)$ be a sequence of functions with
\[
\Ee(w_\e)\le K_1\logep.
\]
Then the Jacobians $J(w_\e)$ are compact in $\dot W^{-1,1}(\Omega)$ and for a subsequence,
\[
J(w_\e)\to J_*=\pi \sum d_i \delta_{a_i},
\]
where $a_i\in\Omega$ are distinct points and $d_i\in \Z\setminus \{0\}$. Furthermore, 
\[
\liminf_{\e\to 0} \frac{1}{\logep} \Ee(w_\e) \ge \pi \sum |d_i| =\|J_*\|_{\mathcal{M}}
\]
with $\|J_*\|_{\mathcal{M}}\le K_1$.
\end{prop}
\begin{proof}
Comparing with the standard Ginzburg-Landau energy $E^1_\e$, we clearly have
\begin{equation}\label{eq:E1vsEeta}
\min_{\overline{\Omega}}\min(\eta_\e^2,\eta_\e^4) E^1_\e(w_\e) \le E^{\eta_\e}(w_\e) \le \max_{\overline{\Omega}}\max(\eta_\e^2,\eta_\e^4) E^1_\e(w_\e) .
\end{equation}
As both the minimum and the maximum can be estimated as $1+O(\frac1\logep)$, we obtain that the
standard Ginzburg-Landau energy is bounded as
\[
E_\e^1(w_\e)\le K_1\logep+ CK_1.
\]
This implies the compactness of $J(w_\e)$, see \cite{JS_Jacobian, SSbook}.

The lower bound for the energy follows from the 
standard Ginzburg-Landau lower bounds and \eqref{eq:E1vsEeta}.
\end{proof}

The following bound is the heart of the second order $\Gamma$-convergence argument.
\begin{prop}\label{prop:locallb}
Let $B_{r_0}(a)\subset \Omega$. If $J(w_\e)\to \pm \pi \delta_\alpha$ in $W^{-1,1}(B_{r_0}(\alpha))$ then 
\[
\liminf_{r\to 0} \liminf_{\e\to 0} \Ee(w_\e;B_r(\alpha)) - \pi \log \frac r \e \ge \gamma_0 + \pi Q_0(\alpha),
\]
where $\gamma_0$ is the constant introduced in \cite[Lemma IX.1]{BBH}. 
\end{prop}
\begin{proof}
We have shown this in Step 4 of the proof of Proposition~\ref{p:excessenergy}. 

\end{proof}
\begin{prop}\label{prop:Gammac2}
Under the assumptions of Proposition~\ref{prop:Gammac1}, if we have for some $K_2>0$  the sharper upper bound
\[
\Ee(w_\e) \le \|J_*\|_{\mathcal{M}} \logep + K_2,
\]
then $J_*=\pi \sum_{i=1}^n d_i\delta_{\alpha_i}$ with distinct $\alpha_i$ and $d_i\in \{\pm 1\}$.  

Furthermore, $(w_\e)$ are then bounded in $W^{1,p}(\Omega)$ for all $p\in[1,2)$ and satisfy the energy lower bound
\[
\liminf_{\e\to 0} \left(\Ee(w_\e) - \pi n \logep\right) \ge n\gamma_0 + W(\alpha_i,d_i) + \pi \sum_{i=1}^n Q_0(\alpha_i),
\]
where $W$ is the renormalized energy defined in \eqref{eq:NeumannW}.

Also, there is $K_3$ such that
\[
\limsup_{\e\to 0}\frac{1}{\e^2} \int_{\Omega} (1-|w_\e|^2)^2 \, dx \le K_3.
\]

\end{prop}
\begin{proof}
We clearly have $E_\e^1(w_\e)\le \|J_*\| \logep + K_4$ with $K_4$ depending on $K_2$ and $\|J_*\|$ as well as on $Q_0$. 
Hence we can use Ginzburg-Landau arguments of Colliander-Jerrard \cite{CJ} or Alicandro-Ponsiglione \cite{AP} to 
obtain the claimed structure of $J_*$.
The $L^p$ bound follows from Proposition~\ref{p:excessenergy}.
Let now $r>0$ be so small that $B_r(a_i)$ are disjoint and do not intersect $\partial\Omega$. Then we apply Proposition
\ref{prop:locallb} to obtain 
\[
\liminf_{r\to 0}\liminf_{\e\to 0}\Ee(w_\e;\bigcup B_r(\alpha_i) ) -\pi n \log \frac r \e -n\gamma_0  -\pi \sum Q_0(\alpha_i)\ge 0. 
\]
 In $\Omega_r(\alpha)$, we let $w_r$ be the optimal $S^1$ valued map as in the definition of the renormalized energy $W$, and 
let $\hat w_*$ be the subsequential weak limit of $w_\e$. Then $|\hat w_*|=1$ a.e. and hence by lower semicontinuity of the Dirichlet integral,
\[
\int_{\Omega_r(\alpha)} |\nabla w_r|^2 dx \le \int_{\Omega_r(\alpha)} |\nabla \hat w_*|^2 dx \le \liminf_{\e\to 0}
\int_{\Omega_r(\alpha)} |\nabla w_\e|^2  dx 
\]
By \eqref{eq:E1vsEeta}, we deduce that also
\[
\liminf_{\e\to 0}\Ee(w_\e;\Omega_r(\alpha)) \ge \frac12 \int_{\Omega_r(\alpha)} |\nabla w_r|^2 dx .
\]
Using that 
\[
W(\alpha,d)=\lim_{r\to 0} \left( \frac12 \int_{\Omega_r(\alpha)} |\nabla w_r|^2 dx -\pi n \log \frac1r \right),
\]
we obtain the claimed lower bound.

The upper bound for the penalty term has been shown in Proposition~\ref{p:excessenergy}.
\end{proof}

\subsection{The gradient flow}
Now we study the gradient flow analog of \eqref{eq:igl},
\begin{equation} \label{eq:gf}
\begin{cases}
\frac{1}{\logep} \p_t u = \Delta u + (p^2(x) - |u|^2) u &\quad\text{in $\Omega$} \\
\frac{\p u}{\p \nu} =0 & \quad\text{on $\p\Omega$}.
\end{cases}
\end{equation}
Instead of deriving an equation of motion by studying the limits of the energy density, we are going to use the method of 
$\Gamma$-convergence of gradient flows. This has the advantage that we will require slightly less regularity and 
convergence properties for the potential term $p(x)$ than in the Schr\"odinger part of this article.
We will  follow closely the approach of Sandier-Serfaty \cite{SSGamma} to obtain the convergence of the gradient flow of $\Ee$ to that of $F$.
In fact, the proof goes through with almost no essential changes, and we will therefore assume the reader is familiar with 
the argument and some of the notation of \cite{SSGamma} and mostly highlight the differences. 

Again, we use the approach of dividing by the profile $\eta_\e$, leading to an equation for $w_\e=\frac{u_\e}{\eta_\e}$. 
As we assume Neumann boundary conditions $\p_\nu \eta_\e=0$ for $\eta_\e$, then Neumann boundary conditions 
 $\p_\nu u=0$ imply Neumann boundary conditions $\p_\nu w_\e=0$. 
As before, we obtain 
\begin{equation}  \label{e:weightedGF}
\begin{cases}
\frac{1}{\logep} \eta^2_\e \p_t w_\e = \dv ( \eta^2_\e \nabla w_\e) + { \eta^4_\e \over \e^2} ( 1 - |w_\e|^2 ) w_\e. &\quad\text{in $\Omega$} \\
\frac{\p w_\e}{\p \nu} =0 & \quad\text{on $\p\Omega$}
\end{cases}
\end{equation}

This is the gradient flow of $\Ee$ with respect to the scalar product on $ \mathcal{X}_\e=L^2(\Omega;\C)$ given by the quadratic form
\[
\|v\|^2_{\mathcal X_\e} =\frac1\logep \int_\Omega \eta_\e^2 |v|^2 dx.
\]
In fact, we have for $\phi \in C^\infty(\Omega;\C)$ the G\^{a}teaux derivative 
\[
\mathrm{d} \Ee (w_\e;\phi) = \int_\Omega \eta_\e^2 \nabla w_\e\cdot \nabla \phi - \frac{\eta_\e^4}{\e^2} (1-|w_\e|^2) w_\e\cdot \phi  dx 
\]
and on the other hand,
\[
\left<\p_t w_\e, \phi \right>_{\mathcal X_\e} = \frac1\logep \int_\Omega \eta_\e^2 \p_t w_\e \phi dx 
\]
so integrating by parts we see that 
\eqref{e:weightedGF} is a gradient flow, more precisely  $\p_t w_\e = -\nabla_{\mathcal{X}_\e} \Ee(w)$. 

We will show that we can relate this gradient flow to that of $F(\alpha)$, where for $\alpha=(\alpha_1,\dots, \alpha_n)\in \Omega^n$, $\alpha_i\neq \alpha_j$ for 
$i\neq j$, we set
\[
F(\alpha)=W(\alpha,d) + \pi \sum_{j=1}^n Q_0(\alpha_j)
\]
and we will use the following metric on $ \mathcal{Y}= \R^{2n}$ (the tangential space to $\Omega^n$):
\[
\|b\|^2_{\mathcal{Y}}=\pi \sum_{j=1}^n |b_j|^2.
\]

\begin{thm}\label{t:GRADFL}
Assume that the profiles $\eta_\e$ solving \eqref{eq:TFeqn} satisfy 
 $Q_\e=\logep (\eta_\e^2-1)\to Q_0$ in $C^1(\overline{\Omega})$.

Let $(w_\e)$ be a family of solutions of \eqref{e:weightedGF} with 
initial data $w_\e^0$.
Assume that $w_\e^0$ satisfy $J(w_\e^0) \to \pi \sum d_j \delta_{ \alpha_j^0}$ in $\dot W^{-1,1}$ with $d_j \in \{\pm1 \}$.

Let $ \alpha (t):[0,T^m)\to \R^{2n}$ be a solution of the system of ODEs 
\begin{equation}\label{e:gfODE}
\pi \dot \alpha_j(t) = -\nabla_{\alpha_j} W(\alpha(t),d) - \pi \nabla Q_0(\alpha_j(t)) 
\end{equation}
where $T^m=T( \alpha^0)\in(0,\infty]$ is the maximal time of existence for \eqref{e:gfODE}, i.e. the smallest time $T$ such that as as $t\to T$, 
we have $r_{\alpha(t)}\to 0$.

If the initial data are very well-prepared, i.e.
$D_\e(0)=\Ee(w_\e^0) - H_\e(\alpha^0)\to 0$ as $\e\to 0$, then for $0\le t<T_m$, 
$J(w_\e(t))\to \pi \sum d_j \delta_{\alpha_j(t)}$ so we have convergence of the gradient flow of $\Ee$ to that of $F$. 
Furthermore, $D_\e(t)=\Ee(w_\e(t))-H_\e(\alpha(t))\to 0$ so the data continue to be well-prepared. 
Finally, we have for $T<T^m$ and any $B_i(t)$ that are disjoint open balls contained in $\Omega$ and centered at $\alpha_i(t)$ that
\begin{equation}
\frac1\logep \int_0^T \int_\Omega \eta_\e^2 \left| \p_t w_\e - \chi_{B_i(t)} \dot \alpha_i(t) \cdot \nabla w_\e \right|^2 dxdt \to 0.
\end{equation}
\end{thm}
\begin{rem}
See \propref{p:ellconv} for conditions on the potential $p_\e$ that imply the convergence of $Q_\e$ assumed in the theorem. Note that
we require less regularity and convergence than for the Schr\"odinger case.
\end{rem}
\begin{proof}[Proof of \thmref{t:GRADFL}]
The proof is mostly an exercise in changing $E^1_\e$ to $\Ee$ in the corresponding results in \cite{SSGamma}. The 
main observation, which we will use repeatedly below,
 is that for any sequence of $f_\e\in L^2(\Omega)$ we have 
\begin{equation}\label{e:obs}
 (1+\frac C \logep)^{-1} \int_\Omega |f_\e|^2 dx\le 
\int_\Omega \eta_\e^2 |f_\e|^2 dx \le (1+\frac C \logep) \int_\Omega |f_\e|^2 dx,
\end{equation}
so in particular if $\int_\Omega \eta_\e^2 |f_\e|^2 dx \le K \logep$, then $\int_\Omega |f_\e|^2 dx \le K\logep + KC$, 
and a similar reverse inequality. 

We note that any solution of \eqref{e:weightedGF} satisfies for $t_1<t_2$
\[
\Ee(w_\e(t_2)) = \Ee(w_\e(t_1)) + \frac{1}{\logep}\int_{t_1}^{t_2} \int_\Omega \eta_\e^2 |\p_t w_\e|^2dxdt
\]
as follows from multiplying \eqref{e:weightedGF} with $\p_t w_\e$ and integrating by parts. In particular, $s\mapsto \Ee(w_\e(s))$ is 
weakly decreasing and 
\[
\int_{t_1}^{t_2}\int_\Omega \eta_\e^2 |\p_t w_\e|^2dxdt\le C \logep^2.
\]

Essentially verbatim as in \cite[Lemma 3.4]{SSGamma}, we can show the existence of $T_0>0$ such that actually 
\[
\int_0^{T_0} \int_\Omega \eta_\e^2 |\p_t w_\e|^2dxdt\le C \logep,
\]
which implies 
\[
\int_0^{T_0} \int_\Omega |\p_t w_\e|^2dxdt\le C \logep+o(\logep).
\]
In addition, we also have for all $t\in[0,T_0)$ with $n=\sum |d_j^0|$ the estimate
$\Ee(w_\e(t))\le \pi n\logep+C $ and so $E_\e^1(w_\e(t))\le \pi n\logep+o(\logep)$. 

We can now apply Proposition 3.2 and Proposition 3.3 of \cite{SSGamma} to $w_\e$ and obtain that 
(for a subsequence) $J(w_\e(t))\to \pi \sum d_j \delta_{b_j(t))}$, where $b_j\in H^1(0,T_0;\R^2)$ with $b_j(0)=\alpha_j^0$, 
i.e. the $d_j$ are constant.
Corollary 7 of \cite{SSprod} now applies to show
\[
\liminf_{\e\to 0} \frac1\logep \int_{t_1}^{t_2}\int_\Omega |\p_t w_\e|^2dxdt \ge \pi \sum \int_{t_1}^{t_2} |\dot b_i|^2 dt .
\]
As $|\eta_\e^2-1|\le \frac C \logep$, this also implies
\begin{equation}\label{e:GCGFpart1}
\liminf_{\e\to 0} \frac1\logep \int_{t_1}^{t_2}\int_\Omega \eta_\e^2 |\p_t w_\e|^2dxdt \ge \pi \sum \int_{t_1}^{t_2} |\dot b_i|^2 dt.
\end{equation}

The last equation is the lower bound part needed for the $\Gamma$-convergence of gradient flows argument. 

The construction also proceeds almost verbatim as in Proposition 3.5 of \cite{SSGamma}. 
Let $w_\e$ satisfy $J(w_\e)\to \pi \sum d_j \delta_{\alpha_j}$ and $D_\e(w_\e,\alpha)\le C$ and assume
that $\|\nabla_{\mathcal{X}_\e} \Ee(w_\e)\|_{\mathcal{X}_\e}\le C$. 
This implies that 
\[
\int_\Omega \frac{1}{\eta_\e^2} \left| \dv (\eta_\e^2 \nabla w_\e)+ \frac{\eta_\e^4}{\e^2}(1-|w_\e|^2) w_\e\right|^2dx \le \frac{C}{\logep}.
\]
We set 
\[f_\e=\dv(\eta_\e^2\nabla w_\e) + \frac{\eta_\e^4}{\e^2}(1-|w_\e|^2)w_\e\]  
and note that $f_\e\to 0$ in $L^2$ and 
 $(iw_\e,f_\e)=(iw_\e,\dv(\eta_\e^2 \nabla w_\e))$. Furthermore
\[
\dv j(w_\e) = (iw_\e,\dv \nabla w_\e ) = (iw_\e,\dv(\eta_\e^2\nabla w_\e)) -\nabla(\eta_\e^2)\cdot j(w_\e) .
\]
By \eqref{e:potbound} we see $|w_\e|\to 1$. From \eqref{e:Lpbound} together with $|\nabla(\eta_\e^2)|\to 0$
we deduce $\dv j(w_\e) \to 0$ in $L^2(\Omega)$. From this point we can continue as in \cite{SSGamma}, as now both
$\curl j(w_\e)$ and $\dv j(w_\e)$ have the same properties as the analogous quantities in \cite{SSGamma}.

Now let $V\in (\R^2)^n$ and let $b(t)$ a curve satisfying $b(0)= \alpha$ and $\p_t b(0)=V$. We claim that there exists a path
$v_\e(t)$ with $v_\e(0)=w_\e$ and the following properties:
\begin{gather}
\|\p_t v_\e(0)\|^2_{\mathcal{X}_\e}  = \|\p_t b(0)\|^2_{\mathcal{Y}}+o(1) \label{eq:kin} , \\
\lim_{\e\to 0} \frac{d}{dt}\big\vert_{t=0} \Ee(v_\e(t)) = \frac{d}{dt}\big\vert_{t=0} F(b(t)) + g( \alpha) D_\e \label{eq:enloss}
\end{gather}
for some function $g$ that is locally bounded in $\{a\in \Omega^n : r_a>0\}$. 

To construct this path in function space, we use the same pushing as in \cite{SSGamma}. Let $B_i=B_\rho( \alpha_i)$ 
with $\rho<r_\alpha$ be pairwise disjoint balls and define $\chi_t:\Omega\to \Omega$ to be 
a one-parameter family of diffeomorphisms that satisfies
\[
\chi_t(x)=x+tV_i \quad\text{in $B_i$.}
\]
With the phase corrector function $\psi_t$ defined as in (3.24) of \cite{SSGamma}, set 
\[
v_\e(\chi_t(x),t) = w_\e(x) e^{i\psi_t(x)}.
\]
The claim \eqref{eq:kin} follows as in \cite{SSGamma}, keeping in mind our observation \eqref{e:obs}.

To calculate the energy change along this path, we change variables $y=\chi_t(x)$ and obtain with $\mathrm{Jac}(\chi_t)$ denoting
the Jacobian determinant,
\begin{align*}
\Ee(v_\e) &= \int_{\Omega} \frac12 \eta_\e^2(y) |\nabla v_\e(y)|^2 + \frac{\eta_\e^4(y)}{4\e^2} (1-|v_\e(y)|^2)^2 dy \\
&= \int_\Omega \left(\frac12 \left( \eta_\e\circ \chi_t(x)\right)^2 \left| D\chi_t^{-1} \nabla (w_\e e^{i\psi_t})\right|^2 
+ \frac{(\eta_\e\circ \chi_t)^4}{4\e^2} (1-|w_\e(x)|^2)^2 |\mathrm{Jac}(\chi_t)|\right) dx.
\end{align*}
We now differentiate in time and set $t=0$. When the derivative does not hit $\eta_\e\circ \chi_t$, the resulting terms 
can be dealt with exactly as in \cite{SSGamma} by our observation \eqref{e:obs}. The remaining terms 
are 
\[
\mathcal{A}_1 = \int_\Omega \frac12 \frac{d}{dt}\Big\vert_{t=0} (\eta_\e\circ \chi_t)^2 |\nabla w_\e|^2 dx
\]
and 
\[
\mathcal{A}_2 = \int_\Omega \frac{1}{\e^2} \frac{d}{dt}\Big\vert_{t=0} (\eta_\e\circ \chi_t)^4 (1-|w_\e|^2)^2 dx.
\]
We note that in $B_i$, $\frac{d}{dt}\big\vert_{t=0} (\eta_\e\circ \chi_t)^2=V_i \cdot \nabla (\eta_\e^2)$. 
Using the uniform convergence 
$\frac1\logep\nabla(\eta_\e^2)\to \nabla Q_0$ and the fact that the logarithmic part of the energy 
is concentrated in the $B_i$ (recall \propref{p:excessenergy} or
(3.17) of \cite{SSGamma}), we deduce that $\mathcal{A}_1\to \pi \sum V_i \cdot \nabla Q_0(\alpha_i)$ as $\e\to 0$.
To show that $\mathcal{A}_2\to 0$, we recall \eqref{e:potbound} and note that $|\nabla (\eta_\e^4)|\to 0$ uniformly in $\Omega$.
Together with the results of \cite{SSGamma} for the terms not involving derivatives of $\eta_\e$,
 we deduce \eqref{eq:enloss}.

On the interval $[0,T_0)$, we can now apply Theorem 1.4 of \cite{SSGamma} and obtain the claim of \thmref{t:GRADFL} up
to the time $T_0$. The global statement up to $T^m$ follows as in Section 3.3 of \cite{SSGamma}.
\end{proof}

\section{Numerical Simulations of the Vortex Dynamics}
\label{s:numerics}

Armed with the dynamical equation \eqref{e:ODE}, we can simulate the flow of vortices and observe the impact of the background potential at the scales we have studied in Theorem \ref{t:vortexdynamics}.  For our simulations, we have used the renormalized energy
\[
W(\alpha,d)=-\pi \sum_{j\neq k} d_j d_k \log|\alpha_j-\alpha_k|,
\]
which is the correct expression for $\Omega=\R^2$, see Remark~\ref{r:extensions}.  This is an approximation for the renormalized energy for $\Omega = B_R (0)$ for $R \gg 1$ sufficiently large compared to $|\alpha|$.  
We will focus here on the case of the dipole (a pair of vortices of opposite charge) interacting with potentials $Q_0=V_i$ of the form
\begin{align}
\label{numpots1}
& V_1 = e^{-|\vec x|^2} \  \ \text{Single Gaussian} , \\
\label{numpots2}
&  V_2 = .225 \tanh (x) \ \ \text{Step Function to different material background} , \\
\label{numpots3}
& V_3 = e^{ - | \vec x - (1,0)|^2} + e^{ - | \vec x + (1,0)|^2}  \ \ \text{Double Gaussian}, \\
\label{numpots4}
& V_4 = \sum_{j,k = -15}^{15}    e^{-|\vec x - (j,k)|^2}  \ \ \text{Lattice of Gaussians}.
\end{align}
While these potentials are not compactly supported, we can apply Theorem~\ref{t:vortexdynamics} by truncating the potentials outside
a suitably large domain without affecting the dynamics we are plotting.

Recall that in the absence of the background potential, dipole dynamics simply move in in a straight line perpendicular to that connecting the vortex centers at a speed correlated to the vortex spacing.  

The equations \eqref{e:ODE} for each choice of background are then plugged into the {\it ode15s} ODE solver in {\it Matlab}  and integrated over time scales long enough to observe the impact of the background potential on the dipole dynamics.    The results are recorded graphically in Figure \ref{dynfig1}.

\begin{figure}[htp]
\begin{tabular}{cc}
\includegraphics[width=4.5cm]{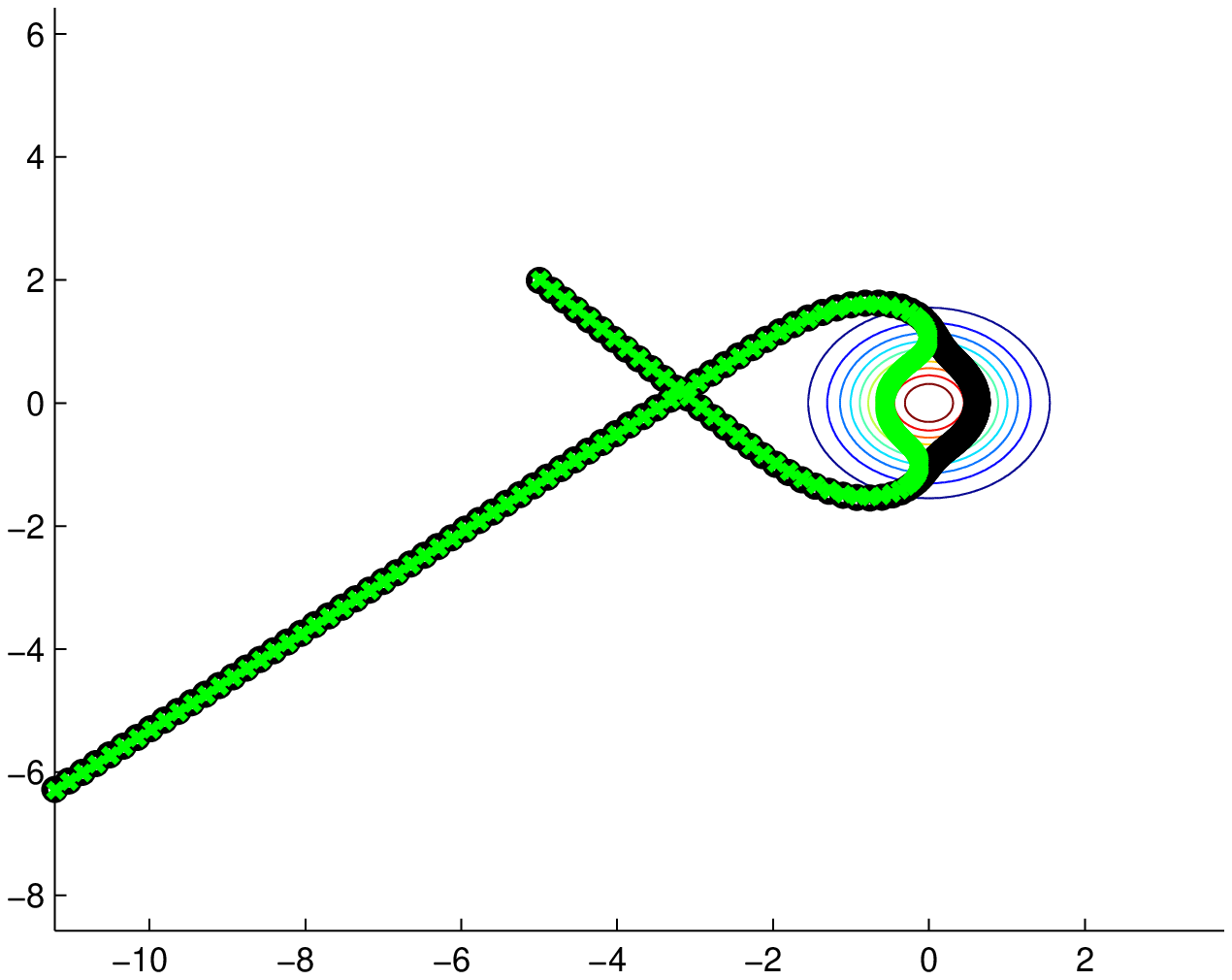} &
\includegraphics[width=4.5cm]{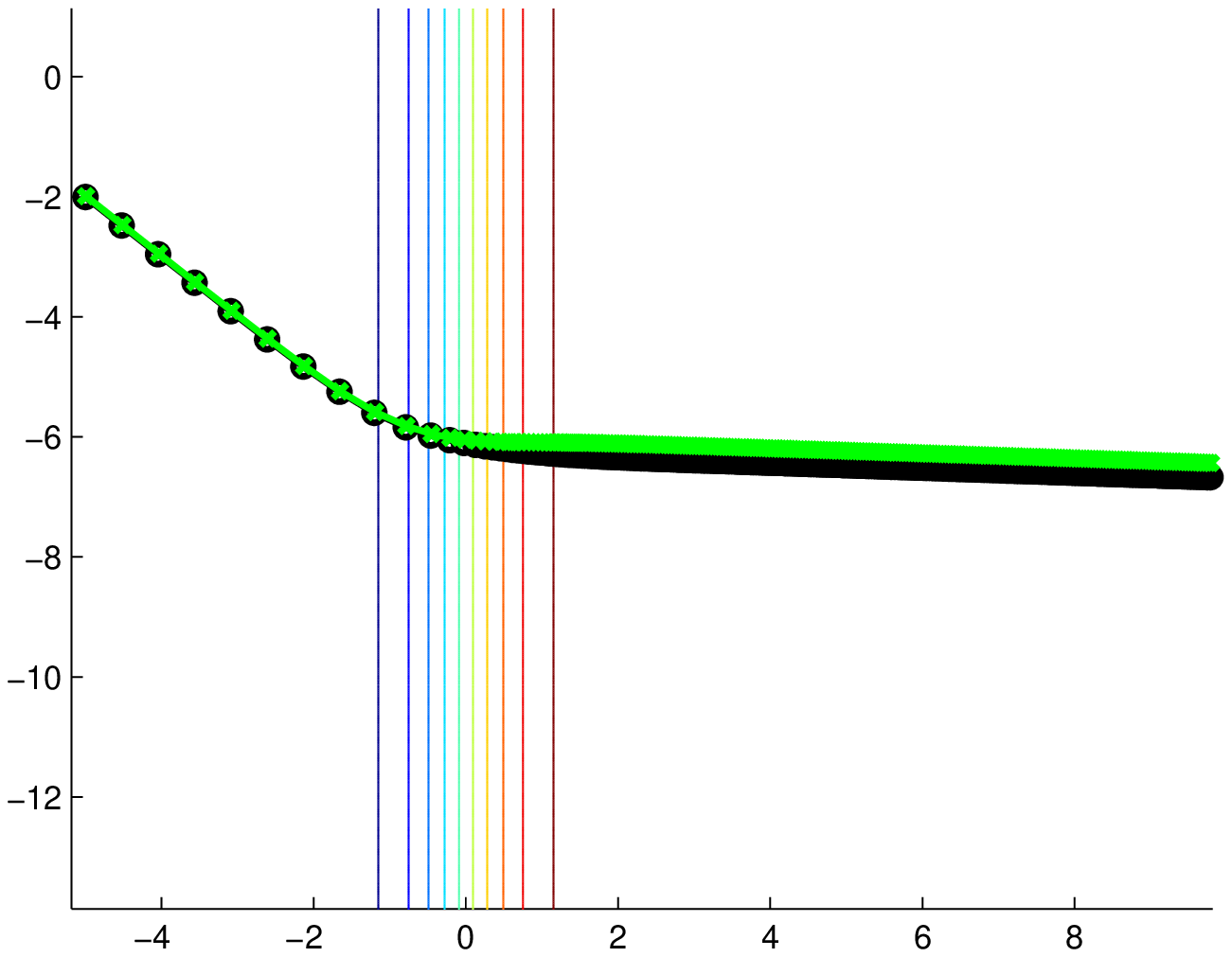} \\
\includegraphics[width=4.5cm]{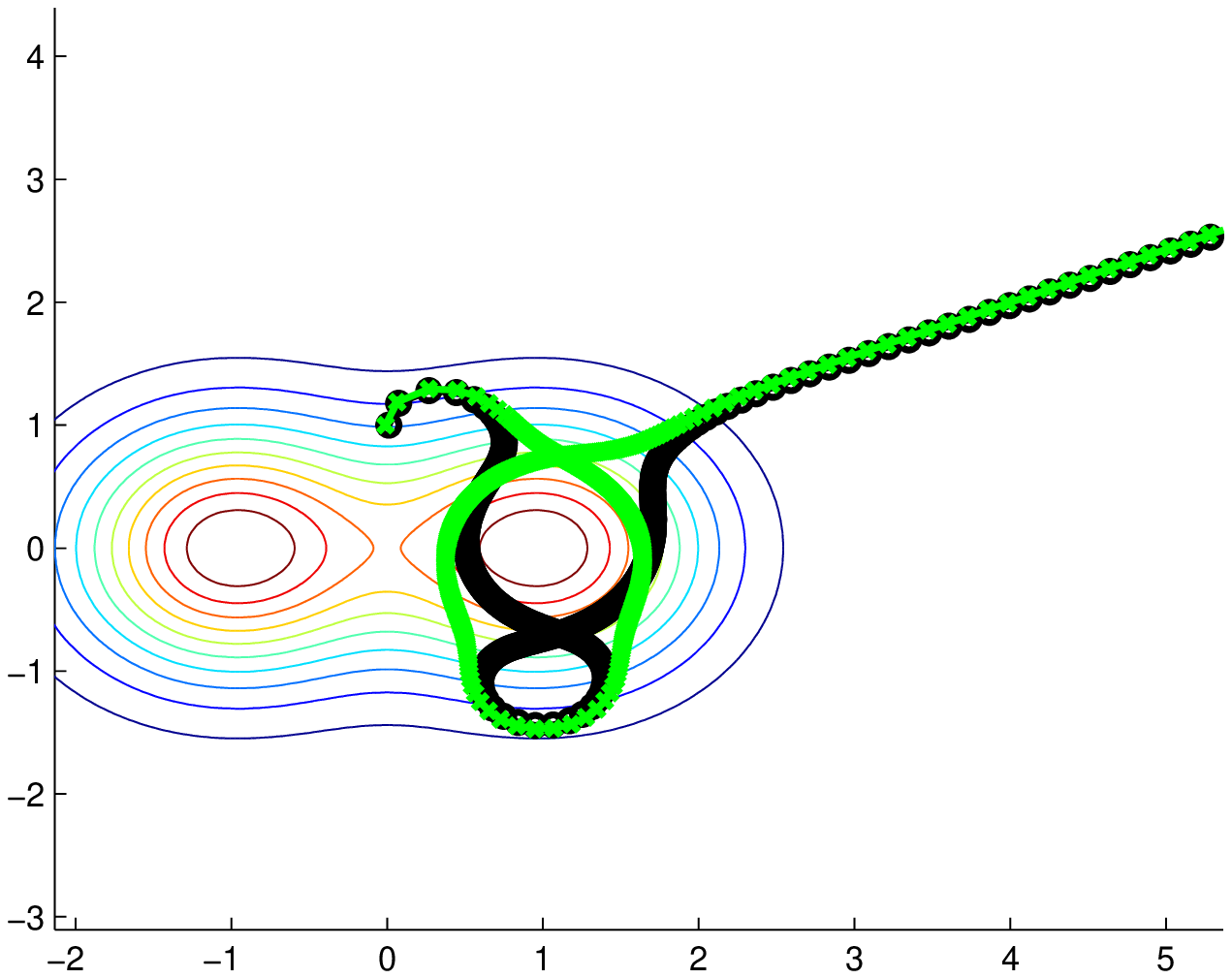} &
\includegraphics[width=4.5cm]{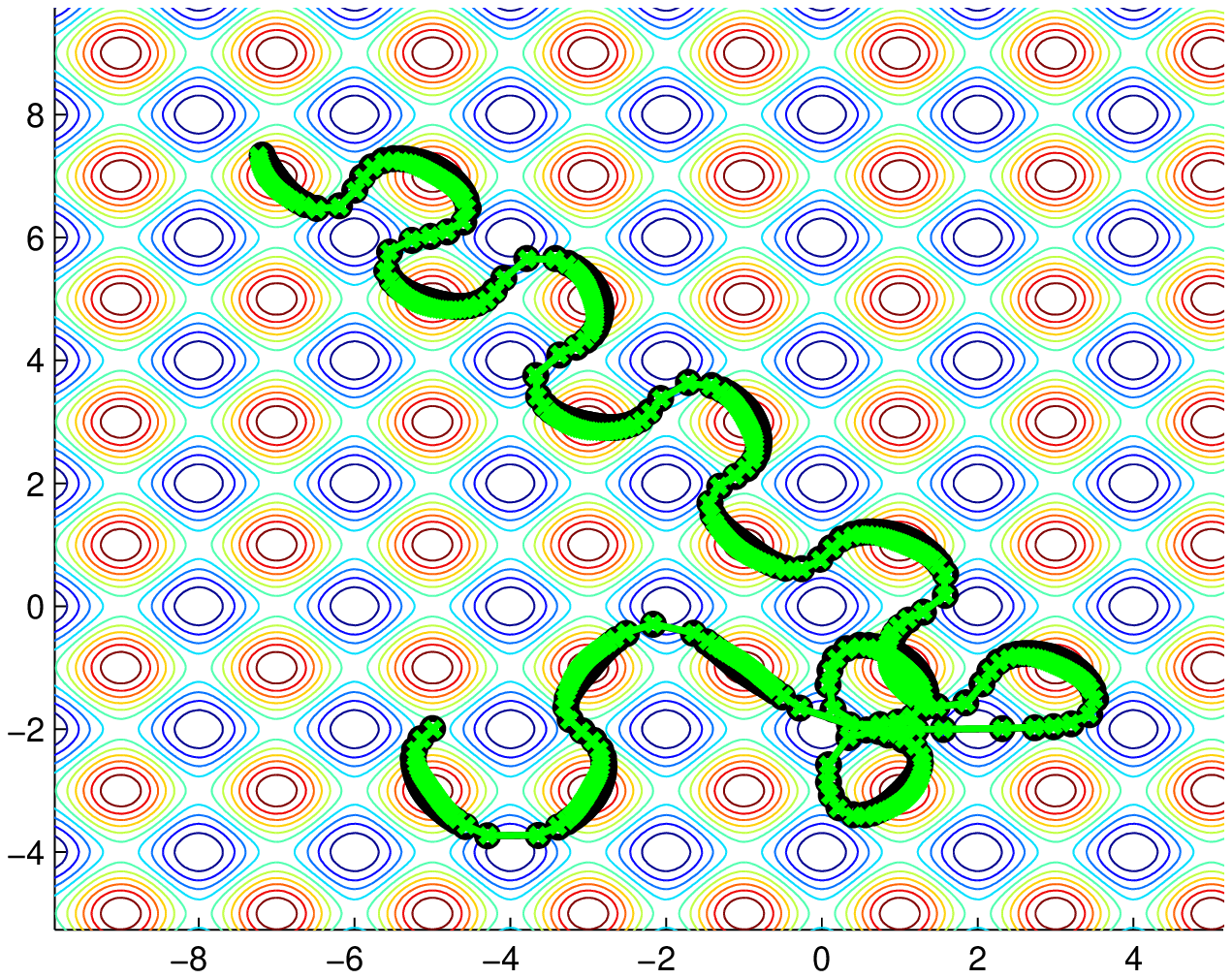} \\
\end{tabular}
\caption{Dipole dynamics for $d_1 = 1$, $d_2 = -1$ plotted over contours of $Q_0$ for the cases of 1. The single Gaussian \eqref{numpots1}
 with $\vec{\alpha}_1 (0) = (-5.0,2.0)$, $\vec{\alpha}_2 (0) = (-4.99,1.99)$ (Top Left), 2. A smooth transition function  \eqref{numpots2} with $\vec{\alpha}_1 (0) = (-5.0,-2.0)$, $\vec{a}_2 (0) = (-4.99,-1.99)$ (Top Right), 3. A symmetric double-well Gaussian potential  \eqref{numpots3} with $\vec{\alpha}_1 (0) = (-0.01,1.0)$, $\vec{\alpha}_2 (0) = (0.01,1.0)$ (Bottom Left), 4. A square lattice array of Gaussians \eqref{numpots4} with $\vec{\alpha}_1 (0) = (-5.0,-1.99)$, $\vec{\alpha}_2 (0) = (-4.99,-2.0)$.}
\label{dynfig1}
\end{figure}

\appendix


\section{Estimates on $\etae$}
\label{s:etacontrol}

We provide here the details required to prove Proposition \ref{p:ellconv}.  We seek to understand bounds on the function $\eta_\e$ defined by
\begin{equation}
\label{eqn:etae}
\Delta \eta_\e = - \frac{1}{\e^2} ( p_\e^2 - \eta_\e^2) \eta_\e,
\end{equation}
with Neumann boundary conditions, where $p_\e$ represents the IGL background on a bounded domain $\Omega \subset \R^2$.

\begin{proof}[Proof of Proposition~\ref{p:ellconv}]

1.  We first claim there exists an $H^{k+2}$ solution for $k \in \mathbb{N}$.   \\

We use a slightly different ansatz for $p_\e$ and $\eta_\e$: 
$$p_\e = 1+\frac{\rt_\e}{|\log\e|}, \ \ \eta_\e=1+\frac{\tQ_\e}{|\log\e|}.$$ 
Recall that we will assume that $p_\e \in H^k$ for $k \geq 2$ with $\e^2  \rt_\e \to 0$ in $H^{k+1}$ as $\e \to 0$ and $\nabla p_0$ is compactly supported strictly on the interior of $\Omega$.  The last assumption in particular will be used for simplicity to allow convergence up the boundary, though this likely can be relaxed.
We first claim that $\eta_\e$ is the unique positive solution of the following minimization problem,
\begin{equation} \label{e:etaenergy}
\eta_\e = \arg \min_{\eta \in H^1(\Omega; \R^1)} \int_\Omega {1\over 2} \LV \nabla \eta \RV^2 + {1\over 4\e^2} \LC p_\e^2 - \eta^2 \RC^2 dx .
\end{equation}
 Since we are working on a bounded set $\Omega$, such a nontrivial $\eta_\e$ exists in $H^1$  for $\e$ sufficiently small and $p_\e$ sufficiently bounded in $H^1$
 using a direct method and Rellich-Kondrachov.
 We can improve the regularity away from the boundary.  In particular if $\rt_\e \in H^k$ then $\eta_\e \in H^{k+2}$, even though the $H^{k+2}$ norm blows up as $\e \to 0$.  This follows from standard elliptic theory results involving nonlinear Sobolev embeddings to bootstrap regularity, see for instance \cite{CMMT}, Section $2.4$. 
  The existence of Euler-Lagrange equations, strong solutions, and classical solutions follow.  

As noted in \eqref{Qe-def} this is a slight shift of notation, as the $Q_0$ defined in the main text is not quite the limit of this family of functions, rather it takes the form
\[ Q_0 =  \lim_{\e \to 0} \left[ 2 \tQ_\e + \frac{\tQ_\e^2}{\logep} \right].\]
 The Euler-Lagrange perturbation equation is written as 
 \begin{align}
 \label{tQeq}
 0 = \e^2 \Delta \tQ_\e + \logep ( 1 + { \tQ_\e \over \logep}) ( { 2 \rt_\e \over\logep} + {\rt_\e^2 \over \logep^2} - {2 \tQ_\e \over \logep} - {\tQ_\e^2 \over \logep^2} ),
 \end{align}
 and so 
\begin{align*}
- {\e^2 \over 2} \Delta \tQ_\e = (\rt_\e - \tQ_\e) \LB \LC 1 + {\tQ_\e \over \logep}  \RC \LC  1 + { \rt_\e \over 2 \logep} + {\tQ_\e \over 2 \logep}      \RC  \RB  .
\end{align*}
 Then we can write this as 
 \begin{align}
 \label{Qeq}
 \LC 1 - \frac{1}{\logep}  \right(  \rt_\e + \frac{\rt_\e^2}{2 \logep} \left)  - {\e^2 \over 2} \Delta \RC  \tQ_\e  = \rt_\e + \frac{\rt_\e^2}{2 \logep} - \frac{\tQ_\e^2}{2 \logep} + \frac{ \tQ_\e^3}{\logep^2}.
 \end{align}

2.  We next establish $\e$-dependent  estimates on a model problem. \\
 
The required estimates follow by, for instance, carefully modifying \cite{Folland}[Theorem $7.32$], which draws upon ideas originally put forth by Nirenberg \cite{Nir1}.  The proof relies upon an analysis of regularity up to the boundary by flattening locally to the half-plane and using the Green's function there, application of difference operators to gain regularity and induction on regularity estimates for elliptic problems with Neumann boundary problems.  Essentially, the argument boils down to the fact that the operator is coercive for each $\e$ however.  The equation also has a unique set of solutions by the same property.\footnote{Perhaps the closest argument of this type for Neumann boundary conditions  can be found in  \cite{Taylor}, Chapter $5.7$, Propositions $7.4$ and $7.5$ where it states that 
for a smooth enough domain $\Omega \subset \RR^2$, there exists a unique solution $v$ to the elliptic equation
\[  (-\Delta + 1) v = f \ \text{in} \ \Omega, \ \ \frac{\partial v}{\partial \nu} = 0 \ \text{on} \ \partial \Omega ,\]
and for all $k = 0,1,2,\dots$, given $f \in H^k$, we have
\[  \| v \|_{H^{k+2} (\Omega)}^2 \leq C \| \Delta v \|_{H^k (\Omega)}^2 + C \| v \|_{H^{k+1} (\Omega)}^2. \]
Since we need precise control  on the constants of our estimate with respect to $\e$ for a perturbation of this equation, which are not available by rescaling, we have included a proof for completeness. }  
 
To establish the bounds we need more precisely, we follow \cite{Folland}[Theorems $7.32$ and $7.29$], where the author studies Sobolev estimates on coercive elliptic equations.  In our setting, these equations take the form
 \[ 
L_\e w := \left(1 + \frac{1}{\logep} A - \frac{\e^2}{2} \Delta \right) w = f, \ \ \partial_\nu w = 0
 \]
 on $\Omega$ smooth enough, for $A$ a bounded function of the same regularity as $p$.

 We will first consider this equation
 on 
 half-balls of radius $s$ with flat boundary, say $N(s)$ where $y$ parametrizes the boundary and $x$ parametrizes the interior.  We claim we have an estimate of the form
  \begin{equation}
 \label{claim}
 \e^2 \| w \|_{H^{k+2}} + \| w \|_{H^k} \leq C \| f \|_{H^k}.
  \end{equation}
 To see this, we first claim that
  \begin{equation}
 \label{claim1}
 \e^2 \| \partial_x^\gamma w \|_{H^1 (N( \tilde r ))} \leq \| f \|_{H^k (N(r))}\
\end{equation}
 with $\tilde r < r$, with $\gamma \leq k + 1$.    For $\gamma = 0$, this is exactly the coercivity estimate.  Otherwise, we use an inductive argument constructed via difference operators inside the Dirichlet form $D(\Delta_h^\gamma \zeta w, \Delta_h^\gamma \zeta w)$ for $\zeta$ a cut-off to $N$ that vanishes on the curved part of the boundary of $N$.  Commuting with 
the cut-off function and integrating by parts when necessary, we observe
\begin{align*}
 & \e^2 \| \Delta_h^\gamma  \zeta w \|_{H^1 (N(\tilde r))} + \left( 1 - \frac{\max |A|}{\logep} \right) \| \Delta_h^\gamma \zeta w \|_{L^2 (N(\tilde r))}  \leq D(\Delta_h^\gamma \zeta w, \Delta_h^\gamma \zeta w)  \\
&  \hspace{2cm} \leq C \e^2 \| \Delta_h^\gamma  \zeta w \|_{H^1 (N(\tilde r))} \left(  \| f \|_{H^k (N(r))} + \| \partial^{\gamma-1} w \|_{L^2} \right),
\end{align*}
where then by induction we have
\begin{align*}
& \e^2 \| \Delta_h^\gamma  \zeta w \|_{H^1 (N(\tilde r))} +  \left( 1 - \frac{\max |A|}{\logep} \right) \| \Delta_h^\gamma \zeta w \|_{L^2 (N(\tilde r))}  \leq D(\Delta_h^\gamma \zeta w, \Delta_h^\gamma \zeta w)  \\
& \hspace{2cm} \leq C  \| f \|_{H^k (N(r))} .
\end{align*}
Similarly, we claim that 
   \begin{equation}
 \label{claim2}
 \| \partial_x^\gamma w \|_{L^2 (N( \tilde r ))} \leq \| f \|_{H^k (N(r))}
 \end{equation}
 for $0 \leq \gamma < k$. This follows by instead putting all the derivatives onto $f$ and using $L^2$ instead of $H^1$ norms for the $f$ term on the right hand side.  
 
 At the boundary, we can then establish \eqref{claim},
 again proven using induction.  To see this, we recognize that if $\gamma$ is a multi-index and $\gamma_2 = 0$ or $1$, then the estimate follows by \eqref{claim1} and \eqref{claim2}.  For
  use the fact that 
 \[ 
 \partial_y^2 w = - \frac{2}{\e^2} \left( f +  \frac{\e^2}{2} \partial_x^2 w - w \right).
 \]
 If the number of derivatives on $y$ is $0$ or $1$, we use the above estimate.  For more than that, we pull off the first two derivatives in $y$, make the substitution, then use the inductive hypothesis.
 
Then, once such estimates are established on half-balls, we can create a sequence of cut-off functions $U_0 = B(0,R), \dots, U_k = V_0$, $\overline{U}_{j+1} \subset U_j$, where $V_0 \cup V_1 \cup \dots \cup V_M = B(0,R)$ and $V_j$ can be mapped to a half-ball for $j > 0$ and $W_0$ is an interior region.  
 
Letting $\zeta_j$ be a cut-off to $U_j$, on the interior we have
\begin{align*}
& \e^2 \| w \|_{H^{j+2}( U_{j+1})} + \left( 1 - \frac{\max |A|}{\logep} \right)  \| w \|_{ H^j (U_{j+1})}  \leq \e^2 \| \zeta_j w \|_{H^{j+2}( U_{j+1})} + \| \zeta_j w \|_{ H^j (U_{j+1})} \\
& \hspace{1cm}  \leq  C \| L \zeta_j w \|_{H^j }  \\
&  \hspace{1.2cm} \leq C \| \zeta_j f \|_{H^j} + \e^2 \| w \|_{H^{j+1} (U_j)} ,
\end{align*}
from which the result follows via induction.  On regions identified with a half-ball, we apply \eqref{claim}.  

3.  We can use the linear estimates \eqref{claim} to generate \emph{a posteriori} estimates on the sequence.  
From the energetic formulation, we have using that $\eta = p_\e$ provides a natural set of bounds the observation
\[
\| p_\e^2 -  \eta_\e^2 \|_{L^2}^2 \leq 2 \e^2 \| \nabla p_\e \|_{L^2}^2.
\]
This gives
\[
\| \rt_\e - \tQ_\e \|_{L^2} \leq \sqrt{2} \e  \| \nabla \rt_\e \|_{L^2}
\]
or
\[
\| \tQ_\e \|_{L^2} \leq C \| \rt_\e \|_{H^1}.
\]
By a similar line of reasoning, we know that 
\[ 
\| \nabla \tQ_\e \|_{L^2} \leq  \| \rt_\e \|_{H^1},
\]
and hence 
\[
\| \tQ_\e \|_{H^1} \leq C  \| \rt_\e \|_{H^1}. 
\]
We observe directly from \eqref{Qeq} that 
\[
\| \tQ_\e \|_{H^k} \leq C  \left( \| \rt_\e + \frac{\rt_\e^2}{2 \logep} \|_{H^k} + \frac{1}{\logep} \| \tQ_\e \|_{H^k}^2 + \frac{1}{\logep^2} \| \tQ_\e \|_{H^k}^3  \right),
 \]
 which gives a uniform control on $\tQ_\e$ in $H^k$ via a boot-strapping argument.   Similarly, we can re-arrange \eqref{Qeq} to observe
 \begin{equation}
 \LC 1 - {\e^2 \over 2} \Delta \RC \LC \tQ_\e - \rt_\e \RC = {\e^2 \over 2} \Delta \rt_\e + {(\rt_\e - \tQ_\e)\over \logep} \mathcal{P}(\tQ_\e, \rt_\e)
 \end{equation}
 where 
 \begin{equation}
 \mathcal{P}(\tQ_\e, \rt_\e) :=  {\rt_\e\over 2} + {3 \tQ_\e \over 2} 
 +  {\LC \tQ_\e ( \rt_\e + \tQ_\e ) \RC \over 2 \logep} .
 \end{equation}
 Hence, applying the linear estimates, we easily observe
\[
\| \tQ_\e - \rt_\e \|_{H^{k-1}} \leq C\e^2 \| \rt_\e \|_{H^{k+1}},
\]
which converges to $0$ as $\e \to 0$ for $\rt_\e$ sufficiently regular.  Note, we have used here that $\nabla \rho_\e$ has compact support in order to integrate by parts.  To remove this condition otherwise would require further work controlling the error terms relating to boundary condition of $\rho_\e$ and potentially restrict us to local convergence estimates.

4.  To finish the \emph{a posteriori} convergence, we need to get to $H^k$ convergence. Recall  the equation
\[
(1 - \e^2 \Delta ) q = g \qquad \p_\nu q = 0 ,
\]
which we rewrite as 
\begin{equation} \label{e:bettereqn}
(q - g ) - \e^2 \Delta q = 0 \qquad \p_\nu q = 0.
\end{equation} 
We will replace $q$ with $\tQ_\e$ and $g$ by the right hand side of our elliptic equation.

For higher regularity, we have the formal calculation, 
\begin{align*}
\int \LV \partial^k (q- g) \RV^2 dx & = \e^2 \int \partial^k (q - g)  \partial^{k} \Delta q dx  \\
& =  - \e^2 \int \LV \nabla \partial^k q \RV^2 + {\e^2 \over 2} \LN \partial^k q \RN_{L^2}^2 + {\e^2 \over 2} \LN \partial^k  g \RN_{L^2}^2  \\
& \qquad + \e^2
\int_{\p \Omega} \p_\nu \partial^k q \LC \partial^k q  - \partial^k g \RC d s ,
\end{align*}
where the boundary terms can be controlled by careful use of the Neumann boundary condition and the equations.

To establish the result rigorously up to the boundary, we must repeat the argument as in Step $2$, but acting $\Delta^\gamma_h$ on both sides of \eqref{e:bettereqn} and then multiplying 
both sides by $ \Delta^\gamma_h(q - g)$ and integrating.   
Specifically, we can consider the coercive Dirichlet form
\[
D( \Delta^k_h \zeta (q- g), \Delta^k_h \zeta (q - g) ) + \e^2 D( \nabla_h \Delta^{k}_h \zeta q, \nabla_h \Delta^{k}_h \zeta q )
\]
both in neighborhoods of the boundary, as well as in the interior to get elliptic estimates as in Step $2$.

To see the uniqueness, let $\eta_j = 1 + {\tQ_j \over \logep}$, $j = \{1,2\}$ be two solutions to \eqref{eqn:etae} with Neumann boundary conditions for the same $p_\e = 1 + {\rt_\e\over \logep}$ with $\rt_\e \in H^k$ for $k > 2$.  Then set $w = \eta_1 - \eta_2$; $w$  solves 
\begin{align*}
\left\{ \begin{array}{rl} 
- \Delta w = {1\over \e^2} \LC w p^2_\e - \LC \eta_1^3 - \eta_2^3 \RC \RC  & \hbox{ in } \Omega \\
\p_\nu w = 0 & \hbox{ on } \p \Omega \end{array} \right. .
\end{align*}
Note that $\eta_1^3 - \eta_2^3 = w \LC \eta_1^2 + \eta_1 \eta_2  +\eta_2^2 \RC = w \LC 3 + {3 \over \logep} ( \tQ_1 + \tQ_2) + {1\over \logep^2} (Q_1^2 + Q_1 Q_2 + Q_2^2 ) \RC$.  Multiply by $w$ and integrate over $\Omega$, using the boundary condition:
\begin{align*}
\int \LV \nabla w \RV^2 + {2 \over \e^2} w^2 & = {1\over \logep \e^2} \int w^2 \LB 2 \rho_\e - 3 (Q_1 + Q_2) +{1\over \logep^2} \LC \rho_\e^2 - \tQ_1^2 - \tQ_1\tQ_2 - \tQ_2^2 \RC \RB \\
& \leq {C\over \logep} \LC 1 + \LN \rho_\e \RN_{H^k}^2 \RC \LC 1 + \LN \tQ_1  \RN_{H^k}^2 \RC\LC 1 + \LN \tQ_2 \RN_{H^k}^2 \RC {1\over \e^2} \int w^2 \\
& \leq  {1.5\over \e^2} \int w^2
\end{align*}
for $\e \leq \e_0$ depending on the norms of $\rt_\e, \tQ_j$, and so $\LN w \RN_{H^1} = 0$.

\end{proof}

\bibliographystyle{abbrv}

\bibliography{kmsbib}

\end{document}